\documentclass[reqno,11pt]{amsart}

\usepackage{amssymb}
\usepackage{mathrsfs}
\usepackage{cite}
\usepackage{amsfonts}

\def\Irr#1{\mathrm{Irr}(#1)}
\def\module#1{\mathrm{mod}\text{-}#1}

\def\wh{\widehat}

\def\MM{\ensuremath{{\mathscr M}}}
\def\pv#1{\ensuremath{{\bf#1}}}

\def\ZZ{\ensuremath{\mathbb Z}}

\def\inv{^{-1}}
\def\p{\varphi}

\def\J{\mathrel{{\mathscr J}}} 

\def\R{\mathrel{{\mathscr R}}} 
\def\Lrel{\mathrel{{\mathscr L}}} 
\def\H{\mathrel{{\mathscr H}}} 

\def\e<{\leq _{E}}

\def\ov#1{\ensuremath{\overline {#1}}}

\def\til#1{\ensuremath{\widetilde {#1}}}

\def\malce{\mathbin{\hbox{$\bigcirc$\rlap{\kern-8.3pt\raise0,50pt\hbox{$\mathtt{m}$}}}}}

\def\1sk{^{(1)}}

\def\to{\rightarrow}

\def\Jbelow#1{{#1}^{\downarrow}}
\def\Jnotup#1{{#1}^{\not\hskip 1pt\uparrow}}
\def\Hom{\mathrm{Hom}}
\def\Ext{\mathrm{Ext}}
\def\Ind{\mathrm{Ind}}
\def\Coind{\mathrm{Coind}}

\def\Thmname{Theorem}
\def\Propname{Proposition}
\def\Lemmaname{Lemma}
\def\Definitionname{Definition}
\newtheorem{Thm}{\Thmname}[section]
\newtheorem{Prop}[Thm]{\Propname}

\newtheorem{Lemma}[Thm]{\Lemmaname}
{\theoremstyle{definition}
}
{\theoremstyle{remark}
\newtheorem{Rmk}[Thm]{Remark}}
\newtheorem{Cor}[Thm]{Corollary}

{\theoremstyle{remark}
\newtheorem{Claim}{Claim}}

\numberwithin{equation}{section}

\title[The homology of regular semigroups]{The quiver of an algebra associated to the Mantaci-Reutenauer descent algebra and the homology of regular semigroups}

\author{Stuart Margolis\ and Benjamin Steinberg}
\address{Department of Mathematics\\
Bar Ilan University\\ 52900 Ramat Gan \\ Israel\and School of Mathematics and Statistics \\ Carleton University \\
1125 Colonel By Drive\\
Ottawa, Ontario  K1S 5B6 \\
Canada}
\thanks{The second author was supported in part by NSERC}
\email{margolis@math.biu.ac.il\and bsteinbg@math.carleton.ca}
\date{November 10, 2008}
\dedicatory{Dedicated to the memory of W.~D. Munn}


\begin{document}
\begin{abstract}
We develop the homology theory of the algebra of a regular semigroup, which is a particularly nice case of a quasi-hereditary algebra in good characteristic.  Directedness is characterized for these algebras, generalizing the case of semisimple algebras studied by Munn and Ponizovksy. We then apply homological methods to compute (modulo group theory) the quiver of a right regular band of groups, generalizing Saliola's results for a right regular band.  Right regular bands of groups come up in the representation theory of wreath products with symmetric groups in much the same way that right regular bands appear in the representation theory of finite Coxeter groups via the Solomon-Tits algebra of its Coxeter complex.  In particular, we compute the quiver of Hsiao's algebra, which is related to the Mantaci-Reutenauer descent algebra.
\end{abstract}
\maketitle

\section{Introduction}
The algebras of (von Neumann) regular semigroups in good characteristic form a wide and natural class of a quasi-hereditary algebras.   Although this class of semigroup algebras may be unfamiliar to many representation theorists, they have surfaced in a number of papers over the past ten years~\cite{renner,Aguiar,Putcharep1,Putcharep5,Putcharep3,Brown1,Brown2,BHR,Saliola,teply,wang}.  Although the fact that these are quasi-hereditary algebras was first pointed out quite late in the game by Putcha~\cite{Putcharep3}, in essence many of the properties of quasi-hereditary algebras were discovered quite early on in semigroup theory in this setting.  For instance, the Munn-Ponizovsky description of the simple modules~\cite{CP,RhodesZalc} are exactly via construction of the standard modules and taking the minimal irreducible constituent in the partial order; the co-standard modules appear in the work of Rhodes and Zalcstein~\cite{RhodesZalc} (which was written in the 1960s); Nico~\cite{Nico1,Nico2} computed early on the bound on the global dimension that one would get from the theory of quasi-hereditary algebras.  The near matrix algebras of Du and Lin~\cite{renewmunn} are essentially the same thing as Munn algebras, introduced by W.~D.~Munn to study semigroup algebras~\cite{CP,RhodesZalc}; in the terminology of~\cite{renewmunn} semigroups algebras of regular semigroups have a bi-free standard system.  In fact, regular semigroup algebras have a canonical quasi-hereditary structure coming from their semigroup structure via a principal series.   Moreover, the associated semisimple algebras in the quasi-hereditary structure are group algebras over maximal subgroups.  So the whole quasi-hereditary structure is already there at the semigroup level.

Just as ordinary group representation theory does not end immediately after observing that group algebras are semisimple, one should not close the door on the representation theory of regular semigroups after observing their algebras are quasi-hereditary.  By passing to the algebra, without remembering the distinguished basis coming from the semigroup, one loses the information that is of interest to a semigroup theorist.  With this philosophy in mind, we do not even give the formal definition of a quasi-hereditary algebra so that workers in semigroups and combinatorics who need to deal with semigroup algebras do not have to assimilate a number of technical definitions from the theory of finite dimensional algebras, which in the context of regular semigroups are quite clear.    In this paper, we take very much the traditional viewpoint in semigroup theory that we want to answer questions modulo group theory.  In particular, we consider, say, the quiver of a semigroup algebra to be computed if we can determine the vertices and arrows modulo being able to compute any representation theoretic fact we need concerning finite groups.  Over algebraically closed fields of characteristic $0$, this assumption is not unreasonable.

Throughout, we make extensive usage of homological methods rather than ring theoretic methods. This is because regular semigroup algebras behave extremely well homologically, especially with respect to taking quotients coming from semigroup ideals, whereas in general it is almost impossible to write down explicitly primitive idempotents for these algebras.

As an application of our techniques, we compute the quiver of a right regular band of groups.  A right regular band is a semigroup satisfying the identities $x^2=x$ and $xyx= yx$.  The faces of a central hyperplane arrangement have the structure of a right regular band, something that  was taken advantage of by Bidigare \emph{et al}~\cite{BHR} and Brown ~\cite{Brown1,Brown2} to compute spectra of random walks on hyperplane arrangements.  See~\cite{Brown1,Brown2} for further examples of applications of right regular band algebras to probability.  Bidigare also discovered that if one takes the reflection arrangement associated to a finite reflection group $W$, then the $W$-invariants of the algebra of the associated right regular band is precisely Solomon's descent algebra; see~\cite{Brown2} for details.   This led Aguiar \emph{et al}~\cite{Aguiar} to develop an approach to the representation theory of finite Coxeter groups via right regular bands.  Saliola computed the quiver of a right regular band algebra and the projective indecomposables~\cite{Saliola}; for the former he used homological methods, whereas in the latter case he computed primitive idempotents.   Actually, all these papers consider the dual notion of left regular bands because they work with left modules.

A right regular band of groups is a regular semigroup in which each right ideal is two-sided.  Intuitively, these are semigroups with a grading by a right regular band so that each homogeneous component is a group.  Hsiao~\cite{Saliolafriend} associates a right regular band of groups $\Sigma_n^G$ to each finite group $G$.  The symmetric group $S_n$ acts by automorphisms on $\Sigma_n^G$ and in the case $G$ is abelian, Hsiao identifies the invariant algebra with the Mantaci-Reutenauer descent algebra~\cite{MantReut} for the wreath product $G\wr S_n$; in the case $G$ is non-abelian, he identifies the invariant algebra with an algebra recently introduced by Novelli and Thibon~\cite{Thibon}.  In this paper, we use homological methods to compute the quiver and the projective indecomposables for the algebra of a right regular band of groups in good characteristic.  As an example, we compute the quiver for the algebra of Hsiao's semigroup $\Sigma_n^G$~\cite{Saliolafriend}.

\section{Preliminaries}
The reader is referred~\cite[Appendix A]{qtheor} or~\cite{CP,Arbib} for background on finite semigroups.  Let $S$ be a finite monoid (in this paper, all monoids and groups are assumed to be finite).  Then $S$ is called \emph{regular} if, for all $s\in S$, there exists $t\in S$ with $sts=s$.  Notice that $st,ts$ are idempotents so in particular every principal left, right and two-sided ideal of a regular semigroup is generated by an idempotent.  Green's preorders are defined by
\begin{itemize}
\item $s\leq_{\J}t$ if $SsS\subseteq StS$;
\item $s\leq_{\R}t$ if $sS\subseteq tS$;
\item $s\leq_{\Lrel}t$ if $Ss\subseteq St$.
\end{itemize}
We write $s\J t$ if $SsS=StS$.  Similar notation is used for $\R$ and $\Lrel$.  One writes $s\H t$ if $s\Lrel t$ and $s\R t$.   Of course $\leq_{\J}$ descends to a partial order on $S/{\J}$.

The set of idempotents of $S$ is denoted $E(S)$.  If $e$ is an idempotent, we write $G_e$ for the group of units of the monoid $eSe$; equivalently $G_e$ is the $\H$-class of $e$.   It is called the \emph{maximal subgroup} of $S$ at $e$.  The following fact about finite semigroups is crucial~\cite[Appendix A]{qtheor}.
\begin{Prop}
Let $S$ be a finite monoid and $e,f\in E(S)$ be $\J$-equivalent idempotents, i.e., $SeS=SfS$.  Then $eSe\cong fSf$ and hence $G_e\cong G_f$.  Moreover, $eS, fS$ ($Se,Sf$) are isomorphic right (resp.~left) $S$-sets.
\end{Prop}

Another important property of finite semigroups is stability~\cite{qtheor}, which states that comparable principal right (left) ideals cannot generate the same two-sided ideal.

\begin{Prop}
Let $s,x$ belong to a finite monoid $S$. Then \[sx\J s \iff sx\R s\ \text{and}\ xs\J s\iff xs\Lrel s.\]
\end{Prop}

It follows easily from this that if $e\in E(S)$ and $J$ is the $\J$-class of $e$, then $J\cap eSe=G_e$, $eS\cap J$ is the $\R$-class of $e$ and $Se\cap J$ is the $\Lrel$-class of $e$.  We use these and other consequences of stability throughout the paper without comment.

\section{The homological theory of regular monoids}
In this section we begin by studying homological aspects of the
algebra of a
regular monoid.  Fix a field $k$ and a regular monoid $S$.  In
characteristic zero, many of the results we present here were deduced by
Putcha~\cite{Putcharep3} as a consequence of the algebra being
quasi-hereditary~\cite{quasihered,drozd}; see also~\cite{teply,wang,Nico1,Nico2}
for results on global dimension. In any characteristic, it is easy to see that
$kS$ is stratified in the sense of~\cite{CPS2} via a principal
series.  However, things are better behaved in general for regular
semigroups and there is no real need to work with principal series.
Moreover, our basic philosophy is to reduce things to computations with
groups.  For these reasons, we provide complete proofs of results that
can be deduced via other methods.

Let $J_1,\ldots,J_n$ be the collection of $\J$-classes of $S$.  Assume that we have ordered them so that $J_i\leq_{\J} J_{\ell}$ implies $i\leq \ell$.  Choose idempotents $e_1,\ldots,e_n$ with $e_i\in J_i$ and let $G_i$ be the maximal subgroup at $e_i$.  Define
\begin{align*}
\Jbelow{J_i} &= \{s\in S\mid s<_{\J} e_i\}\\
\Jnotup{J_i} &= \{s\in S\mid s\ngeq_{\J} e_i\}.
\end{align*}
Both $\Jbelow{J_i}$ and $\Jnotup{J_i}$ are ideals of $S$.  Notice that $\Jbelow{J_i}\subseteq \Jnotup{J_i}$ and $e_i\Jbelow{J_i}=e_i\Jnotup{J_i}$ (and dually).

A key property of regular semigroups is that if $I$ is an ideal, then $I^2=I$ since if $a\in I$ and $aba=a$, then $ba\in I$ and so $a=aba\in I^2$.

If $A$ is an algebra, $\module A$ will denote the category of finitely generated right $A$-modules.
The description of the simple modules for a finite semigroup are well
known, see for instance~\cite{CP,myirreps,Putcharep5,Putcharep3,RhodesZalc}.   We follow here the presentation and ideas of~\cite{myirreps}, which is the shortest and easiest accounting.  First note that by stability, $e_i(kS/k\Jbelow{J_i})e_i\cong kG_i$.  For each $i=1,\ldots, n$, define functors \[\Ind_i,\mathrm{Coind}_i\colon \module {kG_i}\to \module{kS/k\Jnotup{J_i}}\subseteq \module{kS/k\Jbelow{J_i}}\] by
\begin{align*}
\mathrm{Ind}_i(V) = V\otimes _{kG_i}e_ikS/k\Jnotup{J_i}= V\otimes _{kG_i}e_ikS/k\Jbelow{J_i}\\
\mathrm{Coind}_i(V) = \Hom_{kG_i}((kS/k\Jnotup{J_i})e_i,V) = \Hom_{kG_i}((kS/k\Jbelow{J_i})e_i,V).
\end{align*}
These functors are exact and are the respective left and right adjoints of the restriction functor $M\mapsto Me_i$ from $\module{kS/k\Jbelow{J_i}}\to \module{kG_i}$ (in fact $e_ikS/k\Jbelow{J_i}$ and $(kS/k\Jbelow{J_i})e_i$ are free $kG_i$-modules since $G_i$ acts freely on $e_iS\cap J$ and dually).  Also $\Ind_i(V)e_i\cong V\cong \Coind_i(V)e_i$.
The functor $\Ind_i$ preserves projectivity and the functor $\Coind_i$
preserves injectivity as functors to $\module{kS/k\Jbelow{J_i}}$ (but
not in general to $\module{kS}$).  Both functors preserve
indecomposability.

If $V$ is a simple $kG_i$-module, then it is known that $\Ind_i(V)$
has a unique maximal submodule $\mathrm{rad}(\Ind_i(V))$, which is in
fact the largest submodule annihilated by $e_i$ (or equivalently is the submodule of elements annihilated by $J_i$).  The quotient
$\til{V}=\Ind_i(V)/\mathrm{rad}(\Ind_i(V))$ is then a simple
$kS$-module and can be characterized as the unique simple $kS$-module
$M$ such that:
\begin{enumerate}
\item $e_i$ is $\leq_{\J}$-minimal with $Me_i\neq 0$;
\item $Me_i\cong V$ as $kG_i$-modules.
\end{enumerate}
Also one can show that $\til{V}$ is the socle of $\Coind_i(V)$ and can be described as $\Coind_i(V)e_ikS$.  One
calls $J_i$ the \emph{apex} of the simple $kS$-module $\til{V}$.  It
is known that every simple module for $kS$ has an apex, i.e., is of
the form $\til{V}$ for a unique $i$ and a  unique simple $kG_i$-module $V$.   See~\cite{myirreps}.   It is convenient to put a partial order on the simple $kS$-modules by setting $V\leq U$ if $V=U$ or the apex of $V$ is strictly $\J$-below the apex of $U$.  Assume that the characteristic of $k$ does not divide the order of any maximal subgroup of $S$.  Then we call this the \emph{canonical quasi-hereditary structure} on $kS$.  One can show that $kS$ is indeed quasi-hereditary~\cite{quasihered,Putcharep3,drozd} with respect to this partial ordering and that the modules of the form $\Ind_i(V)$ are the standard modules, whereas the modules $\Coind_i(V)$ are the co-standard modules~\cite{Putcharep3}.

If $A$ is an algebra and $I$ an ideal, then $\module{A/I}$ is a full
subcategory of $\module{A}$.  The inclusion has left and right
adjoints given by $M\mapsto M\otimes_{A} A/I$ and $M\mapsto
\Hom_{A}(A/I,M)$ respectively.  If $MI$ denotes the submodule of $M$
generated by all elements $ma$ with $m\in M$ and $a\in I$, then
$M\otimes_{A} A/I \cong M/MI$.  On the other hand
$\Hom_{A}(A/I,M)$ consists of those elements of $M$ annihilated by
$I$.  It is easily verified that the left adjoint preserves projectivity,
the right adjoint preserves injectivity and both are the identity on
$\module{A/I}$~\cite{auslanderidem,Benson}. If $I$ is an ideal of a semigroup $S$ and $M$ is a $kS$-module, we write $MI$ instead of $MkI$ to ease notation.

To compute the quiver of $kS$, we would like to work with the induced
and coinduced modules rather than with projective covers, which we
do not know how to compute in general.  We can do this thanks to the following well-known lemma concerning idempotent ideals~\cite{auslanderidem}, which we prove for completeness.

\begin{Lemma}\label{extstaysthesame}
Let $I$ be an idempotent ideal of an algebra $A$ and suppose that \mbox{$M,N\in \module{A/I}$}.  Then \[\Ext^1_A(M,N)\cong \Ext^1_{A/I}(M,N).\]
\end{Lemma}
\begin{proof}
Let $P$ be a projective $A$-module.  Then the exact sequence of $A$-modules \[0\longrightarrow PI\longrightarrow P\longrightarrow P/PI\longrightarrow 0\] gives rise to the exact sequence
\[\Hom_A(PI,N) \longrightarrow \Ext^1_A(P/PI,N)\longrightarrow 0.\]  But $\Hom_A(PI,N)\cong \Hom_{A/I}(PI/PI^2,N) =0$ as $I^2=I$.  It follows that $\Ext^1_A(P/PI,N)=0$.

Now we can find a short exact sequence of $A/I$-modules \[0\longrightarrow K\longrightarrow (A/I)^m\longrightarrow M\longrightarrow 0.\]  By the above, $\Ext^1_A((A/I)^m,N)=0=\Ext^1_{A/I}((A/I)^m,N)$. Using long exact $\Ext$-sequences,  both $\Ext^1_A(M,N)$ and $\Ext^1_{A/I}(M,N)$ can be identified as the cokernel of the map $\Hom_A((A/I)^m,N)\to \Hom_A(K,N)$.  This proves the lemma.
\end{proof}

The following lemma is from~\cite{auslanderidem}.
\begin{Lemma}\label{strongidempotent1}
Let $A$ be an algebra and $I$ an idempotent ideal that is projective
as a right $A$-module.  Then, for any $A/I$-modules $M,N$, there is an
isomorphism $\Ext^n_{A}(M,N)\cong \Ext^n_{A/I}(M,N)$ all $n\geq 0$.
\end{Lemma}
\begin{proof}
First using the previous lemma, the exact sequence \[0\longrightarrow I\longrightarrow A\longrightarrow A/I\longrightarrow
0\] and the projectivity of $I$ and $A$, we obtain $\Ext^n_A(A/I,N)=0$
for $n\geq 1$ from the long exact $\Ext$-sequence.  Indeed, the case $n=1$ follows from Lemma~\ref{extstaysthesame}.  In general, we have an exact sequence \[0=\Ext^n_A(I,N)\longrightarrow \Ext^{n+1}_A(A/I,N)\longrightarrow \Ext_A^{n+1}(A,N)=0.\]

The lemma is now proved by induction on $n$, the case $n=0$ being trivial and the
case $n=1$ following from Lemma~\ref{extstaysthesame}.  Suppose the
lemma holds for $n$. Again choose a short exact sequence of
$A/I$-modules \[0\longrightarrow K\longrightarrow
(A/I)^m\longrightarrow M\longrightarrow 0.\]  The long exact
$\Ext$-sequence and what we just proved then yield dimension shifts
$\Ext^n_A(K,N)\cong \Ext^{n+1}_A(M,N)$ and $\Ext^n_{A/I}(K,N)\cong
\Ext^{n+1}_{A/I}(M,N)$.  Application of the inductive hypothesis to $K$ completes the proof.
\end{proof}

Our next lemma is a variant on a result of~\cite{auslanderidem} where
filtrations are considered. If $J$ is a $\J$-class and $R$ is an $\R$-class of $S$ contained in $J$, we can make $kJ$ and $kR$ into $kS$-modules by identifying them with the isomorphic vector spaces $kJ+k\Jnotup{J}/k\Jnotup{J}$ and $kR+k\Jnotup{J}/k\Jnotup{J}$, respectively.

\begin{Lemma}\label{exttostrongidemideal}
Let $I$ be an ideal of the regular monoid $S$.  Then we have
$\Ext^n_{kS}(M,N)\cong \Ext^n_{kS/kI}(M,N)$ for any $kS/kI$-modules
$M,N$ and $n\geq 0$.
\end{Lemma}
\begin{proof}
We proceed by induction on the number of $\J$-classes of $I$.  For
convenience, we allow in this proof $I=\emptyset$, in which case the
conclusion is vacuous.  Suppose it is true when the ideal has $m$
$\J$-classes.  Let $I$ be an ideal with $m+1$ $\J$-classes and let $J$
be a maximal $\J$-class of $I$.  Then $I'=I\setminus J$ is an ideal of $S$
with $m$ $\J$-classes. Let $M,N$ be $kS/kI$-modules.  Then they are
also $kS/kI'$-modules and so by induction $\Ext^n_{kS}(M,N)\cong
\Ext^n_{kS/kI'}(M,N)$ for all $n$.  Set $A=kS/kI'$ and $C=kI/kI'$.
Then $A/C\cong kS/kI$ and $C$ is an idempotent ideal of $A$.
Moreover, $C\cong kJ$ and if $e$ is an idempotent of $J$, then $eA=ekJ$ is projective.  Observe that $ekJ=kR$,
where $R$ is the $\R$-class of $e$. Now Green's Lemma~\cite{Green,CP,qtheor,Arbib}  implies $kJ$
is isomorphic to a direct sum of copies of $kR$ as an $A$-module, one for each $\Lrel$-class of $J$, and
so $C\cong kJ$ is projective.  Lemma~\ref{strongidempotent1} then yields
$\Ext^n_{kS/kI'}(M,N)\cong \Ext^n_{kS/kI}(M,N)$, for all $n$,
completing the proof.
\end{proof}

In the terminology of~\cite{auslanderidem}, this result says that $kI$
is a strong idempotent ideal of $kS$.  Let us state a variant of the well-known
Eckmann-Shapiro lemma from homological algebra~\cite{Benson,Hilton}. We sketch the idea of the proof.

\begin{Lemma}[Eckmann-Shapiro]\label{Ecksha}
Let $A$ be an algebra, $e\in A$ an idempotent and $B$ a subalgebra of
$eAe$.  Let $M$ be a $B$-module and $N$ an $A$-module.  If $eA$ is a
flat left $B$-module, then $\Ext^n_A(M\otimes_B eA,N)\cong \Ext^n_B(M,Ne)$.
If $Ae$ is a projective right $B$-module,
$\Ext^n_A(N,\Hom_B(Ae,M))\cong \Ext^n_B(Ne,M)$.  Moreover, these isomorphisms are natural.
\end{Lemma}
\begin{proof}
Since the functor $(-)\otimes_B eA$ is exact (by flatness) and preserves projectives (being left adjoint of the exact functor $V\mapsto Ve$), it takes a projective resolution of a $B$-module $M$ to a projective resolution of the $A$-module $M\otimes_B eA$.  Applying the functor $\Hom_A(-,N)$ to this projective resolution of $M\otimes_B eA$ and using the adjunction gives an isomorphism of the chain complexes computing the $\Ext$-vector spaces $\Ext^n_A(M\otimes_B eA,N)$ and  $\Ext^n_B(M,Ne)$.  The second isomorphism is proved similarly.
\end{proof}

As a consequence we obtain the following natural isomorphisms stemming from induction
and coinduction.

\begin{Lemma}\label{shapiro}
Let $M$ be a $kS/k\Jbelow{J_i}$-module and $V$ a $kG_i$-module, then
\begin{align*}
\Ext^n_{kS}(\Ind_i(V),M) \cong \Ext^n_{kG_i}(V,Me_i)\\
\Ext^n_{kS}(M,\Coind_i(V))\cong \Ext^n_{kG_i}(Me_i,V)
\end{align*}
for all $n\geq 0$.  Moreover, the isomorphisms are natural.
Consequently, the global dimension of $kG_i$ is bounded by the global dimension $kS$.
\end{Lemma}
\begin{proof}
We handle just the first isomorphism as the second is dual.  By
Lemma~\ref{exttostrongidemideal}, $\Ext^n_{kS}(\Ind_i(V),M)\cong
\Ext^n_{kS/k\Jbelow{J_i}}(\Ind_i(V),M)$. Since $e_ikS/k\Jbelow{J_i}$
is a free left $kG_i$-module, the Eckmann-Shapiro Lemma
implies \[\Ext^n_{kS/k\Jbelow{J_i}}(\Ind_i(V),M)\cong
\Ext^n_{kG_i}(V,Me_i).\] The final statement is clear since, for any $kG_i$-module $W$,
\[\Ext^n_{kG_i}(V,W)= \Ext^n_{kG_i}(V,\Ind_i(W)e_i)\cong
\Ext^n_{kS}(\Ind_i(V),\Ind_i(W)).\]  This completes the proof.
\end{proof}

Since a group algebra is well known to have finite global dimension if
and only if the characteristic of the field does not divide the order
of the group, it follows that if $kS$ has finite global dimension then
the characteristic of the field does not divide the order of any
maximal subgroup of $S$.  In fact, Nico~\cite{Nico1,Nico2} showed that
$kS$ has finite global dimension if and only if the characteristic of
$k$ does not divide the order of any maximal subgroup.  More precisely he
proved the following theorem.

\begin{Thm}[Nico]
Let $S$ be a regular monoid and suppose that the characteristic of $k$
does not divide the order of any maximal subgroup of $S$.  If $J$ is a
$\J$-class, define \[\sigma(J) = \begin{cases} 0 & kJ^0\ \text{has an
    identity} \\ 1 & kJ^0\ \text{has a one-sided identity only}\\ 2
  &\text{else.}\end{cases}\]  If $\mathscr C$ is a chain of
$\J$-classes, define $\tau(\mathscr C) = \sum_{J\in \mathscr
  C}\sigma(J)$.  Then the global dimension is bounded by the maximum
of $\tau(\mathscr C)$ over all chains of $\J$-classes of $S$.  In particular,
it is bounded by $2(m-1)$ where $m$ is the length of the longest chain
of non-zero $\J$-classes of $S$.
\end{Thm}

Here $J^0$ is the semigroup with underlying set $J\cup \{0\}$ with $0$ a multiplication given, for $x,y\in J$, by \[x\cdot y = \begin{cases} xy & xy\in J\\ 0 & \text{else.}\end{cases}\]
The final statement of the theorem can also be obtained from the general theory of quasi-hereditary algebras~\cite{quasihered,drozd,Putcharep3}.

The following theorem will allow us to describe to some extent the
quiver of a regular monoid.  One could derive at least a part of it
from the theory of stratified algebras~\cite{CPS2}.  For the
characteristic zero case, Putcha~\cite{Putcharep3} deduced some of these
results from the theory of quasi-hereditary algebras.

\begin{Thm}\label{quiverthm}
Let $S$ be a regular monoid and $k$ a field.
Suppose $\til{U}$ and $\til{V}$ are simple $kS$-modules with apexes $J_i,J_{\ell}$, respectively. Let $N=\mathrm{rad}(\Ind_i(U))$ and $N'=\Coind_i(U)/\til{U}$.   Then:
\begin{enumerate}
\item If $J_i,J_{\ell}$ are $\leq_{\J}$-incomparable, then $\Ext^1_{kS}(\til{U},\til{V})=0$;
\item If $J_i<_{\J} J_{\ell}$, then
\begin{align*}\Ext^1_{kS}(\til{U},\til{V})&\cong \Hom_{kS/k\Jbelow{J_{\ell}}}(N/N\Jbelow{J_{\ell}},\til{V})\\ &\cong \Hom_{kG_{\ell}}([(N/N\Jbelow{J_{\ell}})/\mathrm{rad}(N/N\Jbelow{J_{\ell}})]e_{\ell},V);
\end{align*}
and
\begin{align*}\Ext^1_{kS}(\til{V},\til{U})&\cong \Hom_{kS/k\Jbelow{J_{\ell}}}(\til{V},\Hom_{kS}(kS/k\Jbelow{J_{\ell}},N'))\\ &\cong \Hom_{kG_{\ell}}(V,\mathrm{Soc}(\Hom_{kS}(kS/k\Jbelow{J_{\ell}},N'))e_{\ell});
\end{align*}
\item If $J_i=J_{\ell}$, then $\Ext^1_{kS}(\til{U},\til{V})$ embeds in
  $\Ext^1_{kG_i}(U,V)$, and in particular is $0$ if the characteristic
  of $k$ does not divide $|G_i|$.
\end{enumerate}
\end{Thm}
\begin{proof}
Suppose that $J_i\not>_{\J} J_{\ell}$.   Then the  long exact $\Ext$-sequence derived from the short exact sequence $0\to N\to \Ind_i(U)\to \til{U}\to 0$ yields
\begin{align*}
0&\longrightarrow \Hom_{kS}(\til{U},\til{V})\longrightarrow \Hom_{kS}(\Ind_i(U),\til{V})\longrightarrow \Hom_{kS}(N,\til{V}) \\ & \longrightarrow \Ext^1_{kS}(\til{U},\til{V})\longrightarrow \Ext^1_{kS}(\Ind_i(U),\til{V})
\end{align*}
is exact.
The first non-zero map is an isomorphism since $\til{V}$ is simple and $\til{U}$ is the top of $\Ind_i(U)$. Now $\Hom_{kS}(N,\til{V})=\Hom_{kS}(N/N\Jbelow{J_{\ell}},\til{V})$.  But then
\begin{align*}
\Hom_{kS}(N/N\Jbelow{J_{\ell}},\til{V}) &=
\Hom_{kS}((N/N\Jbelow{J_{\ell}})/\mathrm{rad}(N/N\Jbelow{J_{\ell}}),\til{V})\\
&=\Hom_{kS}((N/N\Jbelow{J_{\ell}})/\mathrm{rad}(N/N\Jbelow{J_{\ell}}),\Coind_\ell(V))
\\
&= \Hom_{kG_{\ell}}([(N/N\Jbelow{J_{\ell}})/\mathrm{rad}(N/N\Jbelow{J_{\ell}})]e_{\ell},V).
\end{align*}
Since $J_i\not>_{\J} J_{\ell}$, $\til V$ is a $kS/k\Jbelow{J_i}$-module.  Lemma~\ref{shapiro} then provides  the isomorphism $\Ext^1_{kS}(\Ind_i(U),\til{V}) \cong \Ext^1_{kG_i}(U,\til{V}e_i)$.

Suppose first that $J_i\neq J_{\ell}$. Then $\til{V}e_i=0$ (as $J_i\ngeq_{\J} J_{\ell}$) and so
\begin{equation*}
0\longrightarrow \Hom_{kS}(N/N\Jbelow{J_{\ell}},\til{V})\longrightarrow \Ext^1_{kS}(\til{U},\til{V})\longrightarrow 0
\end{equation*}
is exact.
Now if $J_i$ and $J_{\ell}$ are incomparable, then $Ne_{\ell}=0$ since $N$ is a $kS/k\Jnotup{J_i}$-module.  This proves (1) and the first statement of (2).  The second statement of (2) is dual.

Assume now $J_i=J_{\ell}$.  Then recalling that $N$ is annihilated by $e_i$, it follows that we have an exact sequence \[0\rightarrow \Ext^1_{kS}(\til{U},\til{V})\longrightarrow \Ext^1_{kS}(\Ind_i(U),\til{V})\cong\Ext^1_{kG_i}(U,\til{V}e_i)=\Ext^1_{kG_i}(U,V)\]
This completes the proof.
\end{proof}

In particular, if the characteristic of $k$ does not divide the order
of any maximal subgroup of $S$, then we have the following result, originally derived by Putcha
when the characteristic of $k$ is zero from the theory of quasi-hereditary algebras~\cite{Putcharep3}.
\begin{Cor}\label{mustcompare}
Let $S$ be a regular monoid and $k$ a field whose characteristic does not divide the order of any maximal subgroup of $S$.  Suppose $\til U$ and $\til V$ are simple $kS$-modules with apexes $J_i$ and $J_{\ell}$, respectively.  Then $\Ext_{kS}(\til U,\til V)\neq 0$ implies that either $J_i<_{\J}J_{\ell}$ or $J_{\ell}<_{\J} J_i$.
\end{Cor}

In the case of characteristic $p$, if $\til U,\til V$ are simple
modules with the same apex $J_i$, it is not necessarily true that
$\mathrm{dim}\ \Ext^1_{kS}(\til U,\til V)= \mathrm{dim}\
\Ext^1_{kG_i}(U,V)$.  For instance, if $C=\{e,g\}$ is a cyclic group
of order $2$ and $k$ is an algebraically closed field of characteristic
$2$, then it is easy to verify that the trivial module $k$ is the
unique simple $kC$-module and $\mathrm{dim}\ \Ext^1_{kC}(k,k)=1$.  On the other hand, let $S$ be the Rees matrix semigroup \[\MM\left(C,2,2,\begin{pmatrix} e & e\\ e& g\end{pmatrix}\right)\] with an adjoined identity. Direct computation shows  $\mathrm{dim}\ \Ext^1_{kS}(\til k,\til k)=0$.

\section{Projective indecomposables and directedness}
A quasi-hereditary algebra is said to be \emph{directed} if all its standard modules are projective.  This depends on the ordering of the simple modules in general.  With this as motivation,  we shall define a regular semigroup $S$ to be \emph{directed} with respect to a field $k$ (or  say that $kS$ is \emph{directed}), if the characteristic of $k$ does not divide the order of any maximal subgroup of $S$ and each induced module $\Ind_i(V)$, with $V$ a simple $kG_i$-module, is projective (i.e., $kS$ is directed with respect to its canonical quasi-hereditary structure).  In this case, it follows that the $\Ind_i(V)$ are the projective indecomposables of $kS$ since $\Ind_i$ preserves indecomposability and $\Ind_i(V)$ has simple top $\til V$.   Our aim is to show that $kS$ is directed if and only if the sandwich matrix~\cite{CP,Arbib,qtheor} of each $\J$-class $J_i$ is left invertible over $kG_i$.
To do this, we prove that $S$ is directed with respect to $k$ if and only if $\til V=\Coind_i(V)$ for all $i$ and all $V$ a simple $kG_i$-module.  This is a standard fact in the theory of quasi-hereditary algebras, but we give a proof for completeness.\footnote{We are grateful to Vlastimil Dlab for pointing out to us that this equivalence is known.}

\begin{Prop}\label{directedcrit}
Suppose that $S$ is a regular monoid and the characteristic of $k$ does not divide the order of any maximal subgroup of $S$.  Then $S$ is directed with respect to $k$ if and only if all the coinduced modules $\Coind_i(V)$ with $V$ a simple $kG_i$-module are simple $kS$-modules (for all $i$).
\end{Prop}
\begin{proof}
Suppose first that $\til V=\Coind_i(V)$ for all simple $kG_i$-modules $V$ and all $i$.  A standard homological argument shows that a module $M$ over a finite dimensional algebra $A$ is projective if and only if $\Ext^1_A(M,W)=0$ for all simple $A$-modules $W$.

So suppose that $W$ is a simple $kS$-module and $V$ is a simple $kG_i$-module.  First assume that the apex $J_\ell$ of $W$ is not strictly $\J$-below $J_i$.  Then $W$ is a $kS/kJ_i^{\downarrow}$-module and so Lemma~\ref{shapiro} yields \[\Ext^1_{kS}(\Ind_i(V),W) = \Ext^1_{kG_i}(V,We_i) =0,\] where the latter equality follows since $kG_i$ is semisimple and hence $V$ is projective.

Thus we may assume that $J_{\ell}<_{\J} J_i$.  By hypothesis, $W=\Coind_\ell(We_{\ell})$.  Since $\Ind_i(V)$ is a $kS/kJ_\ell^{\downarrow}$-module an application of Lemma~\ref{shapiro} implies
\begin{align*}
\Ext_{kS}^1(\Ind_i(V),W) &= \Ext^1_{kS}(\Ind_i(V),\Coind_{\ell}(We_{\ell}))\\ & \cong \Ext^1_{kG_{\ell}}(\Ind_i(V)e_{\ell},We_{\ell})= 0
\end{align*}
where the last equality follows since $\Ind_i(V)e_{\ell}=0$ as $J_{\ell}<_{\J} J_i$ (or by semisimplicity of $kG_{\ell}$).  This proves that $\Ind_i(V)$ is projective.

Suppose conversely that for all $i$ and all simple $kG_i$-modules $V$, one has that $\Ind_i(V)$ is projective.  Then these are precisely the projective indecomposables of $kS$, as discussed above.  Let $U$ be a simple $kG_{\ell}$-module and  set $M$ equal to the cokernel of the inclusion $\til U\to \Coind_\ell(U)$.  If $M=0$, then we are done.  So assume $M\neq 0$.  Since $\Coind_\ell (U)$ is a $kS/k\Jnotup{J_\ell}$-module and $\Coind_\ell(U)e_\ell kS = \til U$, it follows that the apex of any composition factor of $M$ is strictly $\J$-above $J_{\ell}$.  Consequently, the projective cover $P$ of $M$ is a direct sum of modules of the form $\Ind_i(V)$ with $J_i>_{\J} J_{\ell}$.  We shall obtain a contradiction by showing that $\Hom_{kS}(\Ind_i(V),M)=0$ for $J_i>_{\J} J_{\ell}$.  Indeed, by projectivity of $\Ind_i(V)$ we have an exact sequence \[\Hom_{kS}(\Ind_i(V),\Coind_\ell(U))\longrightarrow \Hom_{kS}(\Ind_i(V),M)\longrightarrow 0.\]  Because $\Ind_i(V)$ is a $kS/k\Jbelow{J_{\ell}}$-module the leftmost term of the above sequence is isomorphic to $\Hom_{kG_{\ell}}(\Ind_i(V)e_{\ell},U)=0$ as $J_i>_{\J} J_{\ell}$ implies $\Ind_i(V)e_{\ell}=0$.
\end{proof}

To check the criterion in the above proposition, we need to make explicit how the simple modules of $kS$ sit in the coinduced modules.  What we are about to do is essentially give a coordinate-free argument for the results of Rhodes and Zalcstein~\cite{RhodesZalc}.

\begin{Prop}\label{computesimple}
Let $V$ be a simple $kG_i$-module.  Then there is a natural isomorphism $\Hom_{kS}(\Ind_i(V),\Coind_i(V))\cong \Hom_{kG_i}(V,V)\neq 0$.  Moreover, if $\p\in  \Hom_{kS}(\Ind_i(V),\Coind_i(V))$ is non-zero, then $\ker \p = \mathrm{rad}(\Ind_i(V))$ and $\mathop{\mathrm{Im}}\p = \til V = \mathrm{Soc}(\Coind_i(V))$.
\end{Prop}
\begin{proof}
First note that since $\Ind_i(V),\Coind_i(V)$ are $kS/kJ_i^{\downarrow}$-modules, the adjunction yields \[\Hom_{kG_i}(V,V)=\Hom_{kG_i}(\Ind_i(V)e_i,V)\cong  \Hom_{kS}(\Ind_i(V),\Coind_i(V)).\]  Suppose now that $\p\colon \Ind_i(V)\to \Coind_i(V)$ is a non-zero homomorphism.
Because $\Ind_i(V)e_ikS=\Ind_i(V)$ by construction, it follows that \[\p(\Ind_i(V)) = \p(\Ind_i(V))e_ikS\subseteq \Coind_i(V)e_ikS=\til V.\]  Since $\p\neq 0$, it follows by simplicity of $\til V$, that $\p(\Ind_i(V))=\til V$.  As $\Ind_i(V)$ has a unique maximal submodule, we conclude $\ker \p = \mathrm{rad}(\Ind_i(V))$.
\end{proof}

For a $kG_i$-module $V$, we set $V^* = \Hom_{kG_i}(V,kG_i)$; it is a left $kG_i$-module.  It is a standard fact that $\Hom_{kG_i}(V,W)$ is naturally isomorphic to $W\otimes_{kG_i}V^*$ for finitely generated $kG_i$-modules (when the characteristic of $k$ does not divide $|G_i|$)~\cite{Benson}.  The isomorphism sends $w\otimes \p$ to the map $v\mapsto w\p(v)$.  To prove the isomorphism, one first observes that it is trivial for $V=kG_i$ since both modules are isomorphic to $W$.  One then immediately obtains the isomorphism for all finitely generated projective modules and hence all finitely generated $kG_i$-modules since $kG_i$ is semisimple.

Let $L_i$ and $R_i$ denote the $\Lrel$-class and $\R$-class of $G_i$, respectively.  Observe that as vector spaces, we have $kL_i = (kS/k\Jbelow{J_i})e_i$ and $kR_i = e_ikS/k\Jbelow{J_i}$ by stability.  Moreover, the corresponding $kG_i$-$kS$-bimodule structure on $kR_i$ is induced by left multiplication by elements of $G_i$ and by the right Sch\"utzenberger representation of $S$ on $R_i$~\cite{CP,Arbib,qtheor}  (i.e., the action of $S$ on $R_i$ by partial functions obtained via restriction of the regular action).  To simplify notation, we will use $kR_i$ and $kL_i$ for the rest of this section.  Then
\begin{align*}
\Ind_i(V) &= V\otimes_{kG_i} kR_i \\
\Coind_i(V) &= \Hom_{kG_i}(kL_i,V) = V\otimes_{kG_i} kL_i^*.
\end{align*}

Multiplication in the semigroup induces a non-zero homomorphism
\begin{equation}\label{bilinearmap}
C_i\colon kR_i\otimes_{kS} kL_i\cong e_ikS/k\Jbelow{J_i}\otimes_{kS} (kS/k\Jbelow{J_i})e_i\to e_i(kS/k\Jbelow{J_i})e_i \cong kG_i
\end{equation}
which moreover is a map of $kG_i$-bimodules.   From the isomorphism \[\Hom_{kG_i}(kR_i\otimes_{kS} kL_i,kG_i)\cong \Hom_{kS}(kR_i,\Hom_{kG_i}(kL_i,kG_i))\] we obtain a corresponding non-zero $kS$-linear map $C_i\colon kR_i\to kL_i^*$ (abusing notation).  In fact, $C_i$ is a morphism of $kG_i$-$kS$-bimodules since \eqref{bilinearmap} respects both the left and right $kG_i$-module structures.

Let $T\subseteq R_i$ be a complete set of representatives of the $\Lrel$-classes of $J_i$ and $T'\subseteq L_i$ be a complete set of representatives of the $\R$-classes of $J_i$.  Then $G_i$ acts freely on the left of $R_i$ and $T$ is a transversal for the orbits and dually $T'$ is a transversal for the orbits of the free action of $G_i$ on the right of $L_i$, see~\cite[Appendix A]{qtheor}.  Thus $kR_i$ is a free left $kG_i$-module with basis $T$ and $kL_i$ is a free right $kG_i$-module with basis $T'$.  The dual basis to $T'$ is then a basis for the free left $kG_i$-module $kL_i^*$.  It is instructive to verify that the associated matrix representation of $S$ on $kR_i$ is the classical right Sch\"utzenberger representation by row monomial matrices and the representation of $S$ on $kL_i$ is the left Sch\"utzenberger representation by column monomial matrices~\cite{CP,qtheor,Arbib,RhodesZalc}.    Hence if $\ell_i=|T|$ and $r_i=|T'|$, then as $kG_i$-modules we have $kR_i\cong kG_i^{\ell_i}$ and $kL_i\cong kG_i^{r_i}$.  Thus $C_i$ is the bilinear form given by the $\ell_i\times r_i$-matrix (also denoted $C_i$) with
\begin{equation}\label{structurematrix}
(C_i)_{ba} = \begin{cases} \lambda_b\rho_a & \lambda_b\rho_a\in J\\ 0 & \text{otherwise}\end{cases}
\end{equation}
where $\lambda_b\in T$ represents the $\Lrel$-class $b$ and $\rho_a\in T'$ represents the $\R$-class $a$.  Note that $(C_i)_{ba}\in G_i\cup \{0\}$ by stability and $C_i$ is just the usual sandwich (or structure) matrix of the $\J$-class $J_i$ coming from the Green-Rees structure theory~\cite{CP,Arbib,qtheor}.  The reader may take \eqref{structurematrix} as the definition of the sandwich matrix if he/she so desires.   In particular, using $T$ as a basis for $kR_i$ and the dual basis to $T'$ as a basis for $kL_i^*$, we can view the sandwich matrix of $J_i$ as the matrix of the map $C_i\colon kR_i\to kL_i^*$.  The fact that the sandwich matrix gives a morphism of $kG_i$-$kS$-bimodules translates exactly into the so-called linked equations of~\cite{Arbib,qtheor}.

Putting together the above discussion, we obtain the following module-theoretic version of a result of Rhodes and Zalcstein~\cite{RhodesZalc}.
\begin{Thm}\label{RZsimple}
Let $V$ be a simple $kG_i$-module and suppose that $V$ is flat.  Then the simple $kS$-module $\til V$ is the image of the morphism \[V\otimes C_i\colon \Ind_i(V)=V\otimes_{kG_i} kR_i\to V\otimes_{kG_i}kL_i^*=\Coind_i(V)\] where $C_i$ is the sandwich matrix for $J_i$.  This holds in particular if the characteristic of $k$ does not divide $|G_i|$.
\end{Thm}
\begin{proof}
Since $C_i$ is not the zero matrix and $V$ is flat, $V\otimes C_i$ is a non-zero homomorphism.  Proposition~\ref{computesimple} then implies the desired conclusion.
\end{proof}

We now characterize when $kS$ is directed.  This generalizes the result of Munn and Ponizovsky characterizing semisimplicity of $kS$ in terms of invertibility of the structure matrices since a quasi-hereditary algebra is semisimple if and only if both it and its opposite algebra are directed with respect to a fixed quasi-hereditary structure.

\begin{Thm}\label{isdirected}
Let $S$ be a regular monoid and $k$ a field such that the characteristic of $k$ does not divide the order of any maximal subgroup of $S$.  Then $kS$ is directed if and only if the sandwich matrix of each $\J$-class of $S$ is left invertible.  In this case, the global dimension of $kS$ is bounded by $m-1$ where $m$ is the length of the longest chain of non-zero $\J$-classes of $S$.
\end{Thm}
\begin{proof}
The second statement is immediate from Nico's Theorem.  For the first, we use Proposition~\ref{directedcrit}.  Suppose first that all the sandwich matrices are left invertible.  Let $C_i$ be the sandwich matrix for $J_i$.   By assumption
\begin{equation}\label{model}
kR_i\xrightarrow{\, C_i\,}kL_i^*\longrightarrow 0
\end{equation}
is exact.  Let $V$ be a simple module $kG_i$-module.  Since the tensor product is right exact, tensoring $V$ with \eqref{model} yields the exact sequence
\[\Ind_i(V) = V\otimes_{kG_i} kR_i\longrightarrow V\otimes_{kG_i} kL_i^* = \Coind_i(V)\longrightarrow 0.\]  Since any non-zero homomorphism $\Ind_i(V)\to \Coind_i(V)$ has image $\til V$ by Proposition~\ref{computesimple}, it follows that $\til V=\Coind_i(V)$.  This shows that $kS$ is directed.

Suppose conversely that $kS$ is directed.  Let $C_i$ be the structure matrix of $J_i$.  Since we are dealing with finite dimensional algebras, to show that $C_i$ is left invertible, it suffices to show that $C_i\colon kR_i\to kL^*_i$ is onto.  Let $V$ be a simple $kG_i$-module.  Since $V$ is projective and hence flat, Theorem~\ref{RZsimple} implies that the image of $V\otimes C_i$ is $\til V=\Coind_i(V)$, that is $V\otimes C_i$ is onto.  Because $kG_i$ is semisimple, it follows that if $V_1,\ldots, V_s$ are the simple $kG_i$-modules and $m_1,\ldots, m_s$ are their corresponding multiplicities in $kG_i$, then $C_i = kG_i\otimes C_i=(m_1V_1\otimes C_i)\oplus\cdots\oplus (m_sV_s\otimes C_i)$ and hence is onto.  This completes the proof.
\end{proof}

The following corollary will be useful for computing quivers of directed semigroup algebras.

\begin{Cor}\label{killdowndirectedcase}
Suppose that $kS$ is directed.  Let $\til U,\til V$ be simple modules with respective apexes $J_i\geq _{\J}J_\ell$. Then $\Ext^n_{kS}(\til U,\til V)=0$ all $n\geq 1$.
\end{Cor}
\begin{proof}
Since $\til U$ is a $kS/k\Jbelow{J_{\ell}}$-module and $\til V=\Coind_\ell(V)$, Lemma~\ref{shapiro} yields \[\Ext^n_{kS}(\til U,\til V) = \Ext^n_{kS}(\til U,\Coind_\ell(V)) \cong \Ext^n_{kG_\ell}(\til Ue_\ell,V)=0,\] for all $n\geq 1$, where the last equality follows from semisimplicity of $kG_\ell$.
\end{proof}

Putting together Theorem~\ref{quiverthm} and Corollary~\ref{killdowndirectedcase}, we see that if $kS$ is directed, then $\Ext^1_{kS}(\til U,\til V)\neq 0$ implies that the apex of $\til V$ must be strictly $\J$-above the apex of $\til U$.

We can also compute the Cartan invariants in the case $kS$ is directed, modulo group theory.  Let $k$ be an algebraically closed field.  The \emph{Cartan matrix} $C$ of a finite dimensional $k$-algebra $A$ is the matrix with entries indexed by the projective indecomposables and with $C_{PQ}=\dim\Hom_A(P,Q)$, or equivalently the multiplicity of $P/\mathrm{rad}(P)$ as a composition factor of $Q$~\cite{Benson}.

\begin{Thm}\label{Cartan1}
Let $S$ be a regular monoid and $k$ an algebraically closed field.  Assume that $kS$ is directed.  Let $V$ be a simple $kG_i$-module and $W$ a simple $kG_\ell$-module.  Set $P=\Ind_i(V)$, $Q=\Ind_\ell(W)$.  Then \[C_{PQ} = \begin{cases} \dim \Hom_{kG_i}(V,\Hom_{kS}(kS/k\Jbelow{J_i},Q)e_i) & J_i>_{\J} J_\ell \\ 1 & V=W\\ 0 &\text{otherwise.}\end{cases}\]  In particular, the Cartan matrix $C$ of $kS$ is a unipotent matrix.
\end{Thm}
\begin{proof}
First note that since $Q$ is a $kS/k\Jnotup{J_\ell}$-module, all of its composition factors have apex $\geq_{\J} J_{\ell}$.  Suppose that $J_i=J_{\ell}$.  Then $\Hom_{kS}(P,Q) = \Hom_{kG_i}(V,Qe_i) = \Hom_{kG_i}(V,W)$ and so in this case $C_{P,Q}=1$ if $V=W$ and $0$ otherwise.  Finally, suppose that $J_i>_{\J} J_{\ell}$.  Then since $P$ is a $kS/k\Jbelow{J_i}$-module
\begin{align*}
\Hom_{kS}(P,Q)&\cong \Hom_{kS/k\Jbelow{J_i}}(\Ind_i(V),\Hom_{kS}(kS/k\Jbelow{J_i},Q))\\ &\cong \Hom_{kG_i}(V,\Hom_{kS}(kS/k\Jbelow{J_i},Q)e_i)
\end{align*}
as required.  The final statement follows since we can define a partial order on the projective indecomposables by $\Ind_i(V)\geq \Ind_\ell(W)$ if and only if $J_i=J_{\ell}$ and $V=W$ or $J_i>_{\J} J_{\ell}$.  Hence the Cartan matrix is unitriangular with respect to an appropriate ordering of the projective indecomposables.\end{proof}

This theorem will be the starting point for a more detailed computation of the Cartan invariants for the algebra of a right regular band of groups (Theorem~\ref{Cartan2}).

\section{Right regular bands of groups}
In this section we specialize our results to right regular bands of
groups.  A \emph{right regular band
  of groups} (which we shall call an RRBG) is a regular semigroup for
which each right ideal is a two-sided ideal, or equivalently Green's
relations $\J$ and $\R$ coincide.  In particular, RRBGs include groups, right
regular bands, commutative regular semigroups and right simple
semigroups.

\subsection{The structure of right regular bands of groups}
Here we record the basic structural properties of right regular bands of groups.
For basic facts about finite semigroups, the reader is referred to~\cite[Appendix A]{qtheor} or to~\cite{Arbib,Almeida:book,CP}.
 If $s$ belongs to a finite
semigroup $S$, then $s^{\omega}$ denotes the unique idempotent
positive power of $s$.

\begin{Prop}\label{RRBGprops}
A semigroup $S$ is a right regular band of groups if and only if it satisfies
\begin{enumerate}
\item $s^{\omega}s=s$
\item $s^{\omega}ts^{\omega}=ts^{\omega}$
\end{enumerate}
all $s,t\in S$.
\end{Prop}
\begin{proof}
Suppose first that $S$ satisfies the above two properties.  By (1), each element of $S$ generates a cyclic group and so $S$ is clearly regular.  Suppose $R$ is a right ideal of $S$ and let $r\in R$ and $s\in S$.  Let $a\in S$ be such that $rar=r$.  Then $ra$ is an idempotent, and so by (2) $sr=srar=rasrar\in R$. Thus $R$ is a two-sided ideal.

Conversely, suppose $S$ is an RRBG.  Since $S$ is regular, $s=sts$
for some $t\in S$. The element $e=ts$ is an idempotent $\Lrel$-equivalent to
$s$.  On the other hand, $eS=SeS=SsS=sS$ and so $e\R s$.  Thus the
$\H$-class of $s$ is a group~\cite{qtheor,Arbib} and so
$s^{\omega}s=s$.

For (2), choose $n>1$ so that $x^n=x^{\omega}$ all $x\in S$.  Then we
have \[ts^{\omega}=(ts^{\omega})^{\omega}ts^{\omega}= (ts^{\omega})^{n-1}t(s^{\omega}ts^{\omega})\] and so
$s^{\omega}ts^{\omega}\Lrel ts^{\omega}$.  Since $\mathscr J=\mathscr R$,
it follows $s^{\omega}ts^{\omega}\R ts^{\omega}$.  But then
$ts^{\omega}\in s^{\omega}S$ and so $s^{\omega}ts^{\omega} =
ts^{\omega}$, as required.
\end{proof}

In particular, it follows that the class of RRBGs is closed under
product, subsemigroups and quotients.  Also a band is an RRBG if and
only if it is a right regular band in the sense that is satisfies the
identities $x^2=x$ and $xyx=yx$.
Recall that $E(S)$ denotes the idempotent set of $S$.

\begin{Cor}
Let $S$ be an RRBG.  Then $E(S)$ is a right regular band.
\end{Cor}
\begin{proof}
By Proposition~\ref{RRBGprops}, if $e,f$ are idempotents, then
$efef=eef=ef$.  Thus $E(S)$ is a subsemigroup and hence a right
regular band.
\end{proof}

\begin{Cor}\label{retracttoe}
Let $S$ be an RRBG and $e\in E(S)$, then the map $s\mapsto se$ gives a
retraction from $S$ to $Se=eSe$.
\end{Cor}
\begin{proof}
Indeed, from $ete=te$ we obtain $ste=sete$.
\end{proof}

The next lemma shows that the set of principal right ideals of an RRBG $S$ is
a meet semilattice grading $S$ in a natural way. For the case of a
right regular band (monoid), Brown and Saliola
call the dual of this lattice the support
lattice~\cite{Brown1,Brown2,Saliola}.  The lemma is a special case of a
result of Clifford~\cite[Chapter 4]{CP}.

\begin{Lemma}\label{grading}
Let $S$ be a right regular band of groups.  Then $aS\cap bS=abS$.  Hence the principal right ideals form a meet semilattice $\Lambda$ and the map $a\mapsto aS$ is a surjective homomorphism $S\to \Lambda$ whose fibers are the $\R$-classes.
\end{Lemma}
\begin{proof}
Let $e=a^{\omega}$ and $f=b^{\omega}$.  Then $eS=aS$, $fS=bS$ and
$abS=SabS=SefS=efS$.    Now by Proposition~\ref{RRBGprops}, $efS=fefS$
and so $efS\subseteq eS\cap fS$.  On the other hand, if $x\in eS\cap
fS$, then $ex=x, fx=x$ and so $efx=x$ and hence $x\in efS$.  Thus
$abS=efS=eS\cap fS=aS\cap bS$.
\end{proof}

Consequently, we have the following well-known result.

\begin{Cor}\label{jnotup}
Let $S$ be a RRBG and let $J$ be $\J$-class.  Then $S\setminus \Jnotup
J$ is a subsemigroup of $S$.
\end{Cor}

\subsection{The algebra of a right regular band of groups}
Fix a monoid $S$ which is an RRBG and let $k$ be an algebraically
closed field such that the characteristic of $k$ does not divide the
order of any maximal subgroup of $S$. We retain the notation from the previous sections.  Since each $\J$-class $J_i$ of $S$ consists of a single $\R$-class, the sandwich matrix of each $\J$-class $J_i$ consists of a single column with non-zero entries from $G_i$ and hence is automatically left invertible over $kG_i$.  Therefore, $kS$ is directed by Theorem~\ref{isdirected}.  Hence the global dimension of $kS$ is bounded by the longest chain of non-zero $\J$-classes minus $1$ and the projective indecomposables are precisely the induced modules $\Ind_i(V)$ with $V$ a simple $kG_i$-module.  In this case, they have a particularly simple form.  If $J$ is a $\J$-class of $S$ with maximal subgroup $G$ and $V$ is a simple $kG$-module, then the associated projective indecomposable is $V\otimes_{kG}kJ$ where $s\in S$ acts on $kJ$ as multiplication by $s$ if $J\subseteq SsS$ and as the zero map if $J\nsubseteq SsS$.
The simple modules are the coinduced modules.

\subsubsection{The semisimple quotient}
For a right regular band, the semisimple quotient of its algebra is the algebra of a semilattice cf.~\cite{Brown1,Brown2}.  In the case of a right regular band of groups, one can replace the semilattice with a semilattice of groups, as was observed in~\cite{AMSV,Mobius1,Mobius2}.  The notion of a semilattice of groups is extremely important in semigroup representation theory.  For instance, it is shown in~\cite{AMSV} that a semigroup $S$ has a basic algebra over an algebraically closed field $k$ if and only if it admits a homomorphism to a semilattice of abelian groups inducing the semisimple quotient on the level of their semigroup algebras.  All this will be made explicit here for the case of RRBGs as we will need the details in order to compute the Cartan invariants explicitly.  We begin by recalling the notion of a semilattice of groups, a construction going back to Clifford~\cite{CP}.

Let $\Lambda$ be a  meet semilattice and $G\colon \Lambda^{op}\to \pv{Grp}$ be a presheaf of groups on $\Lambda$, that is, a contravariant functor from $\Lambda$ (viewed as a poset) to the category of groups.  If $e\leq f$, and $g\in G(f)$, then we write $g|_e$ for the image of $g$ under restriction map  $G(f)\to G(e)$.  Then one can place a semigroup structure on $T=\coprod_{e\in \Lambda} G(e)$ by defining the product of $g\in G(e)$ and $h\in G(f)$ by $gh = g|_{ef} h|_{ef}$~\cite{CP}.  Such a semigroup is called a \emph{semilattice of groups}.  Clifford showed that semilattices of groups are precisely the regular semigroups with central idempotents, or alternatively the inverse semigroups with central idempotents~\cite[Chapter 4]{CP}.  Recall that a semigroup $S$ is called \emph{inverse} if, for all $s\in S$, there exists a unique $s^*\in S$ with $ss^*s=s$ and $s^*ss^*=s^*$; equivalently  $S$  is inverse if it is regular and has commuting idempotents~\cite{CP}.  It is a well known result of Munn and Ponizovsky that if $S$ is an inverse semigroup and $k$ is a field such that the characteristic  of $k$ divides the order of no maximal subgroup of $S$, then $kS$ is semisimple~\cite{CP}.  In the case that $T$ is a semilattice of groups as above, it can be shown that $kT\cong \prod_{e\in \Lambda} kG(e)$ using a M\"obius inversion argument~\cite{Mobius1,Mobius2} .

Now let $S$ be again our fixed right regular band of groups and $k$ our algebraically closed field of good characteristic.  Let $\Lambda = S/{\J}=S/{\R}$ be the lattice of $\J$-classes, which is isomorphic to the lattice of principal right ideals cf.~Lemma~\ref{grading}.  We retain the notation of the previous sections: so $\Lambda=\{J_1,\ldots, J_n\}$ and we have fixed idempotents $e_1,\ldots,e_n$ representing the $\J$-classes with corresponding maximal subgroups $G_1,\ldots,G_n$.  Define $i\wedge \ell$, for $i,\ell\in \{1,\ldots,n\}$, by  the equation $J_i\wedge J_\ell=J_{i\wedge \ell}$.  We form a semilattice of groups $F\colon \Lambda^{op}\to \pv{Grp}$ by setting $F(J_i)=G_i$ and defining the restriction $G_i\to G_\ell$, for $J_\ell\leq_{\J} J_i$ by $g|_{J_\ell} = ge_\ell$.  It is immediate from stability and Corollary~\ref{retracttoe} that this restriction map is a homomorphism and that if $i=\ell$, then it is the identity map.  If $J_i\geq_{\J}J_{\ell}\geq_{\J} J_m$, then $e_\ell\geq_{\R}e_m$ and so $e_\ell e_m=e_m$.  Thus $(g|_{J_\ell})|_{J_m} = ge_\ell e_m = ge_m = g|_{J_m}$, establishing functoriality.  Let $T=\coprod_{J_i\in \Lambda} G_i$ be the corresponding inverse monoid.

There is a surjective homomorphism $\p\colon S\to T$ given by $\p(s)=se_i$ if $s\in J_i$.  Indeed, from Lemma~\ref{grading}, one has if $s\in J_i$ and $t\in J_\ell$, then $st\in J_{i\wedge \ell}$.  Moreover, denoting by $\cdot$ the product in $T$ and using $e_i,e_{\ell}\geq_{\R} e_{i\wedge\ell}$, we obtain \[\p(s)\cdot \p(t)=se_i\cdot te_j =se_ie_{i\wedge\ell}te_\ell e_{i\wedge\ell}=se_{i\wedge\ell}te_{i\wedge\ell} = ste_{i\wedge\ell}=\p(st)\] establishing that $\p$ is a homomorphism.  Suppose that $\p(s) = e_i$ for $s\in S$.  Then $s\in J_i$ and $se_i=e_i$.  Since $e_i\R s^{\omega}$, we have $s=ss^{\omega} = se_is^{\omega} = e_is^{\omega}=s^{\omega}$.  Thus $\p^{-1}(e_i)$ is the right zero semigroup $E(J_i)$.  It follows from~\cite[Theorem 3.5]{AMSV} that the induced surjective map $\Phi\colon kS\to kT$ has nilpotent kernel.  Since $kT\cong kG_1\times\cdots \times kG_n$ is semisimple, we conclude that $\Phi$ is the semisimple quotient.  One can describe the semisimple quotient directly by defining $\psi\colon kS\to kG_1\times\cdots \times kG_n$ by $\psi(s) = (g_1,\ldots,g_n)$ where \[g_i = \begin{cases} se_i & s\geq_{\J} J_i\\ 0 & \text{otherwise.}\end{cases}\]  See~\cite{Mobius1,Mobius2} for details.  Notice that $kS$ is basic if and only if each of its maximal subgroups is abelian.  In~\cite{AMSV}, the semigroups with basic algebras were determined for any field.

\subsubsection{The Cartan invariants}
Assume now that $k$ is an algebraically closed field of characteristic $0$.  Using the character formulas for multiplicities from~\cite{Mobius2}, we compute the Cartan invariants of $kS$.  We retain the above notation.  Let $\mu$ be the M\"obius function~\cite{Stanley} for the lattice $\Lambda$ of $\J$-classes.  Let $V$ be a simple $kG_i$-module and let $M$ be any $kS$-module.  Let $\chi_V$ be the character of $V$ and $\theta$ the character of $M$.  Then since $kT$ is the semisimple quotient of $kS$, it follows that $\theta$ factors through $\p$ as $\chi\p$ with $\chi$ the character of $M/\mathrm{rad}(M)$ as a representation of $T$.  Observing that the semilattice of idempotents $E(T)$ of $T$ is isomorphic to $\Lambda$, it follows from the formula in~\cite{Mobius2} for multiplicities of irreducible constituents in representations of inverse semigroups that the multiplicity of $\til V$ as a composition factor of $M$ is given by the formula
\begin{equation}\label{multformula}
\frac{1}{|G_i|}\sum_{g\in G_i}\chi_V(g\inv)\sum_{J_m\leq_{\J} J_i} \theta(ge_m)\mu(J_m,J_i).
\end{equation}

We apply \eqref{multformula} to compute the Cartan invariants for $kS$.  Let $P$ and $Q$ be projective indecomposables for $kS$.  We already know from Theorem~\ref{Cartan1} that $P= \Ind_i(V)$ and $Q=\Ind_\ell(W)$ for appropriate simple $kG_i$ and $kG_\ell$-modules $V$ and $W$ since $kS$ is directed. Moreover, the entry $C_{PQ}$ of the Cartan matrix is $0$ unless $J_i\geq_{\J} J_{\ell}$ and that if $J_i=J_{\ell}$, then $C_{PQ}=1$ if $V=W$ and $0$ if $V\neq W$.  All that remains then is to compute $C_{PQ}$ in the case $J_i>_{\J} J_{\ell}$.

\begin{Thm}\label{Cartan2}
Let $S$ be a right regular band of groups and $k$ an algebraically closed field of characteristic $0$.  Let $V$ be a simple $kG_i$-module and $W$ a simple $kG_\ell$-module with respective characters $\chi_V,\chi_W$. Set $P=\Ind_i(V)$ and $Q=\Ind_\ell(W)$.  Denote by $\mu$ the M\"obius function for the lattice of principal right ideals of $S$ . Let $C$ be the Cartan matrix of $kS$.  Then
 \[C_{PQ} = \begin{cases} \ast & J_i>_{\J} J_\ell \\ 1 & V=W\\ 0 &\text{otherwise.}\end{cases}\]
where \[\ast= \frac{1}{|G_i|}\sum_{g\in G_i}\chi_V(g\inv)\sum_{J_\ell\leq_{\J} J_m\leq_{J}J_i}\mu(J_m,J_i)\sum_{e\in E(J_\ell), (ege_m)^{\omega}=e}\chi_W(ege_{\ell}).\]
\end{Thm}
\begin{proof}
It only remains to consider the case that $J_i>_{\J} J_\ell$ by the remarks before the theorem.
The Cartan invariant $C_{PQ}$ is exactly the multiplicity of $\til V$ as a composition factor of the module $Q$.  Let $\theta$ be the character of $Q$.  Since $Q=W\otimes_{kG}kJ_\ell$, it follows that if $s\ngeq_{\J} J_\ell$, then $\theta(s)=0$.  Hence in our setting, the second sum in \eqref{multformula} can be taken over those $J_m$ with $J_{\ell}\leq_{\J} J_m\leq_{\J} J_i$.  To compute $\theta$, we observe that we can take $E(J_{\ell})$ as a set of representatives of the $\Lrel$-classes of $J_\ell$.  Then $kJ_{\ell}$ is a free left $kG_{\ell}$-module with basis $E(J_{\ell})$.  Let $b=|E(J_{\ell})|$.  The isomorphism of $kJ_\ell$ with $kG_\ell^b$ sends $x\in J$ to the row vector with $xe_\ell\in G_{\ell}$ in the coordinate indexed by $x^{\omega}$ and $0$ in all other coordinates.  The associated matrix representation of $S$ over $kG_{\ell}$ takes $s\geq_{J} J_{\ell}$ to the row monomial matrix $RM(s)$  which has its unique non-zero entry in the row corresponding to $e\in E(G_{\ell})$ in the column corresponding to $(es)^{\omega}$ and this entry is $ese_\ell\in G_{\ell}$.  This is just the Sch\"utzenberger representation by row monomial matrices over $G_\ell$~\cite{CP,Arbib,qtheor,RhodesZalc}.  Now if $\rho\colon G_\ell\to GL(k)$ is the irreducible representation afforded by $W$, then the matrix  $\rho\otimes RM(s)$ for $s$ acting on $W\otimes_{kG_\ell} kJ_\ell$ is obtained by applying $\rho$-entrywise to $RM(s)$~\cite{RhodesZalc}.  The block row $e$ has a diagonal entry if and only if $(es)^{\omega}=e$ and this entry is $\rho(ese_\ell)$.  Hence the trace of $\rho\otimes RM(s)$, that is $\theta(s)$, is given by  \[\sum_{e\in E(J_{\ell}), (es)^{\omega}=e}\chi_W(ese_{\ell}).\]  The formula for $C_{PQ}$ now follows from \eqref{multformula} and the observation  \mbox{$J_m\geq_{\J} J_\ell$} implies $e_m\geq_{\R}e_\ell$ and so $e_me_\ell =e_\ell$.
\end{proof}

In the case that $S$ is a right regular band, the modules $V$ and $W$ are trivial and the above formula reduces to Saliola's result~\cite{Saliola}.

\subsection{The quiver of a right regular band of groups}
Set $A=kS$ and let $e\in E(S)$.  From Proposition~\ref{RRBGprops},  we easily deduce the following lemma.

\begin{Lemma}\label{prop2algebra}
Let $e\in E(S)$ and $a\in A$.  Then $eae=ae$.  In particular, $eAe=Ae$ and the map $a\mapsto ae$ is a retraction from $A$ onto $eAe=Ae$.
\end{Lemma}

Let us recall that the restriction functor $M\mapsto Me$ from $\module{A}$ to
$\module{eAe}$ admits a right adjoint, the functor $V\mapsto
\Hom_{eAe}(Ae,V)=\wh{V}$.  Since $eAe=Ae$ by the above lemma, as a
vector space  $\Hom_{eAe}(Ae,V)$ is just $V$ (identify $f$ with
$f(e)$). The computation \[(fa)(e) = f(ae)=f(eae) = f(e)eae\] shows that
the module action of $A$ on $V$ is given by $va = veae$, i.e., one
makes $V$ an $A$-module via the retraction in
Lemma~\ref{prop2algebra}. In particular, this action extends the
action of $eAe$ and hence $V$ is simple if and only if $\wh{V}$ is
simple.  Evidently the functor $V\mapsto \wh{V}$ is exact and
$\wh{V}e=V$.  The following is a special case of Green's theory~\cite[Chapter 6]{Greenpoly}.

\begin{Lemma}\label{babygreen}
Let $A=kS$ with $S$ an RRBG and let $e\in E(S)$.
Up to isomorphism, the simple $A$-modules $M$ with $Me\neq 0$ are  the
modules of the form $\wh{V}$ with $V$ a simple $eAe$-module. More
precisely, if $M$ is  a simple $A$-module with $Me\neq 0$, then $Me$ is a simple $eAe$-module and $M=\wh{Me}$.
\end{Lemma}
\begin{proof}
Assume $M$ is a simple $A$-module with $Me\neq 0$.  Let $m\in
Me$ be non-zero.  Then $meA = mA=M$, so $meAe=Me$.  Thus $Me$ is
simple.  Now $\Hom_A(M,\wh{Me}) =
\Hom_{eAe}(Me,Me)\neq 0$.  Thus $M\cong \wh{Me}$ by simplicity.
\end{proof}

Since $Ae=eAe$ is a free $eAe$-module, the Eckmann-Shapiro Lemma (our Lemma~\ref{Ecksha})
immediately yields.
\begin{Lemma}\label{shapiro2}
If $M$ is an $A$-module and $N$ is an $eAe$-module, then there is an
isomorphism $\Ext^n_A(M,\wh{N})\cong \Ext_{eAe}^n(Me,N)$ for all
$n\geq 0$.
\end{Lemma}

In particular, we obtain the following corollary.

\begin{Cor}\label{cutdowntomonoid}
Let $A=kS$ with $S$ an RRBG and let $e\in E(S)$.  Suppose $M,N$ are
$A$-modules with $N$ simple and $Ne\neq 0$.  Then $\Ext^n_A(M,N)\cong \Ext^n_{eAe}(Me,Ne)$.
\end{Cor}
\begin{proof}
By Lemma~\ref{babygreen}, $N\cong \wh{Ne}$ and so $\Ext^n_A(M,N)=\Ext^n_A(M,\wh{Ne})\cong \Ext^n_{eAe}(Me,Ne)$, the last isomorphism coming from Lemma~\ref{shapiro2}.
\end{proof}

Assume now that $k$ is an algebraically closed field.
The (Gabriel) quiver of $kS$ is the directed graph with vertex set the
simple $kS$-modules and with $\dim \Ext^1_{kS}(\til{U},\til{V})$ arrows
from  $\til{U}$ to $\til{V}$.
 Suppose $\til{U}$ and $\til{V}$ have apexes $J_i,J_\ell$,
 respectively. By Corollary~\ref{quiverthm} and Corollary~\ref{killdowndirectedcase}, there are no arrows
 $\til{U}\to \til{V}$ unless
 $J_i<_{\J} J_{\ell}$.  Since $e_i<_{\J} e_{\ell}$ and $\til{U}e_i\neq
 0$, it follows that $\til{U}e_{\ell}\neq 0$. Also
 $\til{V}e_{\ell}\neq 0$.  Hence Corollary~\ref{cutdowntomonoid} implies that to compute the number of arrows between $\til{U}$ and
 $\til{V}$ (in both directions) we may assume that $S$ is a monoid and
 $\til{V}$ has apex the unit group.
By Corollary~\ref{jnotup}, $T=S\setminus \Jnotup {J_i}$ is a submonoid of $S$
(necessarily a RRBG) and clearly $kT\cong kS/k\Jnotup {J_i}$.  As a
consequence of Lemma~\ref{exttostrongidemideal}, it follows that to
compute $\Ext^1_{kS}(\til{U},\til{V})$, we may assume that $J_i$ is the
minimal ideal of $S$ (note that we may have to replace $e_i$ with the idempotent $e_ie_{\ell}$ in our computations when cutting down to $Se_{\ell}$).

Summarizing, we have reduced the computation of the quiver of a right
regular band of groups $S$ (in good characteristic) to the following
situation:  we have a simple module $\til{U}$ with apex the minimal
ideal $J$ of $S$ and a simple module $\til{V}$ with apex the
unit group $G$ of $S$ where, moreover, $G\neq J$.    Our goal is to compute
$\Ext^1_{kS}(\til U,\til V)$.

Let $e$ be an idempotent in $J$ and set $H=eJe$; it is the maximal
subgroup at $e$.  Then $eS=J$ and $Se=H$, that is, $J$ is the $\R$-class of $H$ and $H$ is its own $\Lrel$-class.   We write $\Ind$ and $\Coind$ for the induction and
coinduction functors adjoint to the restriction $\module{kS}\to
\module{kH}$ given by $M\mapsto Me$.  As $kH=kSe$, we have $kH^*=\Hom_{kH}(kSe,kH) = \Hom_{eAe}(Ae,kH)=\wh{kH}$.  Hence we can identify $kH$ with $kH^*$ as a vector space.  The $kH$-$kS$-bimodule structure is given by the regular left action of $kH$ and by the right action $xa=xeae$ for $x\in kH$ and $a\in kS$.

\begin{Prop}\label{gettingthere}
The map  $f\colon kJ\to kH$ given by $f(x) = xe$ is a surjective homomorphism of $kH$-$kS$-bimodules.  The kernel $N$ has basis consisting of all differences of the form $x-xe$ with $x\in J\setminus H$.
\end{Prop}
\begin{proof}
It follows immediately from Lemma~\ref{prop2algebra} that $f$ is a
morphism of bimodules.  Clearly it is onto since it restricts to the
identity on $kH$. Morover, $kJ\cong N\oplus kH$ and the
splitting $kJ\to N$ sends $x$ to $x-xe$.  In particular, the elements
of the form $x-xe$ with $x\in J\setminus H$ span $N$.  Since
$\mathrm{dim}\ N= |J|-|H|$, the above set is indeed a basis.
\end{proof}

\begin{Rmk}
The map $f$ is easily verified to be the map given by the sandwich matrix of $J$ if we choose to represent each $\Lrel$-class of $J$ by its idempotent.
\end{Rmk}

Let us set $I=S\setminus G$.  Then $I$ is an ideal of $S$.   Our next
goal is to identify $N/NI$ as a $kH$-$kG$ bimodule. The approach is inspired by Saliola~\cite{Saliola}.  Define a relation $\smallsmile$ on $J$ by $x\smallsmile x'$ if:
\begin{enumerate}
\item $xe=x'e$; and
\item there exists $w\in I\setminus J$ with $xw=x$ and $x'w=x'$.
\end{enumerate}
Let $\approx$ be the equivalence relation generated by
$\smallsmile$. It follows easily that if $x\approx x'$, then $xe=x'e$.
In particular, distinct elements of $H$ are never equivalent. Notice that if $S=J\cup G$, then $\smallsmile$ relates no elements.  We remark that in (2), one may always assume $w$ is idempotent by replacing it with $w^{\omega}$.

\begin{Prop}\label{biaction}
If $g\in G$ and $x\smallsmile x'$, then $xg\smallsmile x'g$.  If $h\in H$, then $hx\smallsmile hx'$
\end{Prop}
\begin{proof}
Suppose $x\smallsmile x'$.  By definition $xe=x'e$ and $xw=x, x'w=x'$ some $w\in I\setminus J$.  Then $xge=xege=x'ege=x'ge$. Also, since $g\inv wg\in I\setminus J$, we have $xg(g\inv wg)=xwg=xg$ and $x'g(g\inv wg)=x'g$.    For the second statement, evidently $hxe=hx'e$ and $hxw=hx,hx'w=hx'$.  This establishes the proposition.
\end{proof}

Let us set $X=J/{\approx}$.  It follows from
Proposition~\ref{biaction} that $X$ admits a left action of $H$ and a right action of $G$ that commute and hence $kX$ is a $kH$-$kG$ bimodule.  There
is an epimorphism $\epsilon\colon kX\to kH$ of $kH$-$kG$
bimodules given by $[x]\mapsto xe$ where $[x]$ denotes the
$\approx$-class of $[x]$. Surjectivity comes about as $\epsilon ([h])=h$ for $h\in H$; it is a morphism of right $kG$-modules since $xge=xege$.  Let $M=\ker \epsilon$.  Then $M$ is also a
$kH$-$kG$ bimodule.  By Corollary~\ref{retracttoe}, there is a
homomorphism $\psi\colon G\to H$ given by $\psi(g)=ge$.  The character of
$M$ as a $kG$-module is the permutation character associated to the action of $G$ on $X$
minus the character of $kH$ viewed as a right $kG$-module via $\psi$.
It is easy to see that $M$ has a basis consisting of the elements of the form $[x]-[xe]$ where $[x]$ runs over all equivalence classes not containing an element of $H$. In particular, $\mathrm{dim}\ M=|X|-|H|$.   Let $N$ be as in Proposition~\ref{gettingthere}.

\begin{Lemma}\label{thebigone}
As $kH$-$kG$ bimodules, $N/NI\cong M$.
\end{Lemma}
\begin{proof}
Let $\Omega$ be a transversal to $\approx$ containing $H$ and let
$\ov{x}$ be the element of $\Omega$ equivalent to $x$ for $x\in J$.
Define a linear map $T\colon N\to M$ on the basis by $x-xe\mapsto [x]-[xe]$.
This is clearly a surjective map since the elements of the form $[x]-[xe]$ span $M$.
Notice that if $x,x'\in J$ with $xe=x'e$, then $T(x-x')=T(x-xe)
+T(x'e-x') = [x]-[xe] +[x'e]-[x'] =[x]-[x']$.  Consequently, if $h\in
H$ and $g\in G$, then as $ge=ege$ \[T(h(x-xe)g) = T(hxg-hxeg) = [hxg]-[hxeg] =
h[x]g-h[xe]g\] so $T$ is a bimodule homomorphism.

Next, we establish $NI\subseteq \ker T$.  Indeed, if $w\in I$, then we
claim that, for any $x\in J$, either $xw=xew$ or $xw\smallsmile xew$. Clearly
$xwe=xewe$.  Suppose that
$w\in J$.  Then $xw\H xew$ by stability and so
$(xw)^{\omega}=(xew)^{\omega}$, whence
\[xw=xw(xw)^{\omega}=xw(xew)^{\omega}=xew(xew)^{\omega}=xew\] where the penultimate equalty uses $ze=eze$ all $z\in S$. If
$w\notin J$, then $xw = xww^{\omega}$, $xew=xeww^{\omega}$ and
$w^{\omega}\in I\setminus J$, establishing $xw\smallsmile xew$.  Thus
$T((x-xe)w) = [xw]-[xew]=0$.  Elements of the form $(x-xe)w$ with $x\in J, w\in I$ span $NI$,
yielding the inclusion  $NI\subseteq \ker T$.  To show that the induced map $T\colon N/NI\to M$ is an isomorphism, we show that $\mathrm{dim}\ N/NI\leq \mathrm{dim}\ M$.  To this effect, we show that $N/NI$ is spanned by elements of the form $x-xe+NI$ with $x\in \Omega\setminus H$.  The number of such elements is $\mathrm{dim}\ M$.

We begin by showing that if $x\approx x'$, then $x-x'\in NI$.  Suppose first
$x\smallsmile x'$.  Then since $xe=x'e$, we have $x-x'\in N$.  Also
there exists $w\in I\setminus J$ so that $xw=x,x'w=x'$.  Then
$x-x'=(x-x')w\in NI$. In general, there exist $x=x_0\smallsmile
x_1\smallsmile x_2\smallsmile \cdots\smallsmile x_n=x'$.  Then
\[x-x'=(x_0-x_1)+(x_1-x_2)+\cdots+(x_{n-1}-x_n)\in NI,\] as
required. Next suppose  $x\approx x'$ and $u\approx u'$.  Then we
claim $x-u+NI=x'-u'+NI$.  Indeed, $x-u = x-x'+x'-u'+u'-u$ and
$x-x',u'-u\in NI$, as was already observed.  In particular, for $x\in
J$ we have $x-xe+NI = \ov{x}-\ov xe+NI$ (since $\ov xe=xe$).  Because
elements of the form $x-xe+NI$ span $N/NI$, we are done.
\end{proof}

It follows that $U\otimes_{kH} N/NI\cong U\otimes_{kH} M$ as
$kG$-modules.  We are now ready to finish up our computation of the quiver of an RRBG.

\begin{Thm}\label{mainquiverthmrrbg}
Retaining the previous notation, $\mathrm{dim}\ \Ext^1_{kS}(\til U,\til V)$ is the multiplicity of $V$ as a composition factor in the $kG$-module $U\otimes_{kH} M$.
\end{Thm}
\begin{proof}
We have an exact sequence of $kH$-$kS$ bimodules \[0\longrightarrow
N\longrightarrow kJ\longrightarrow kH\longrightarrow 0.\]
Since $kH$ is semisimple, all $kH$-modules are projective and hence
flat and so there results an exact sequence of $kS$-modules
\[0\longrightarrow U\otimes_{kH} N\longrightarrow U\otimes_{kH}kJ\longrightarrow U\otimes_{kH} kH\longrightarrow 0.\]  Recall that $H$ is its own $\Lrel$-class, $J$ is the $\R$-class of $H$ and we are identifying $kH^*$ with $kH$.  Thus the middle term is $\Ind(U)$, whereas the rightmost term is $\Coind(U)=\til U$ (since $kS$ is directed).    Because $\Ind(U)$ has simple top, it follows that $\mathrm{rad}(\Ind(U))=U\otimes_{kH} N$. Theorem~\ref{quiverthm} implies in our setting
\begin{align*}
\Ext^1_{kS}(\til U,\til V)&= \Hom_{kG}(\mathrm{rad}(\Ind(U))/\mathrm{rad}(\Ind(U))I,V)\\
&= \Hom_{kG}(\mathrm{rad}(\Ind(U))\otimes_{kS} kS/kI,V)\\
&= \Hom_{kG}(U\otimes_{kH} N\otimes_{kS}kS/kI,V)\\
&=  \Hom_{kG}(U\otimes_{kH} N/NI,V)\\
&\cong \Hom_{kG}(U\otimes _{kH} M,V)
\end{align*}
which is the desired multiplicity as $kG$ is semisimple and $k$ is algebraically closed.
\end{proof}

\section{Examples}
From the results of the previous section it follows that one can in principle compute the quiver over the complex numbers of a right regular band of groups provided one has the character tables of its maximal subgroups.  The algorithm reduces to computing the number of arrows from a simple module with apex the minimal ideal to a simple module with apex the group of units.  This section provides a number of examples. We retain throughout the notation of the previous section.

\subsection{Right regular bands}
Our first example is the case of a right regular band $S$.  In this case, if $U$ is a simple module with apex  the minimal ideal and $V$ is a simple module with apex the singleton $\J$-class of the identity, then all the groups involved are trivial and the representations are trivial.  Thus the number of arrows from $U$ to $V$ is just $\dim M=|X|-1$.  This is exactly Saliola's result~\cite{Saliola}.

\subsection{Permutation groups with adjoined constant maps}
Next suppose that $G\leq S_n$ is a permutation group of degree $n$. Let $\ov G$ consist of $G$ along with the constant maps on $\{1,\ldots, n\}$.  Then $\ov G$ is an RRBG with group of units $G$ and minimal ideal $J$ consisting of the constant maps.  The Krohn-Rhodes Theorem~\cite{PDT,Arbib,qtheor,Eilenberg} implies that every finite semigroup is a quotient of a subsemigroup of an iterated wreath product of semigroups of the form $\ov G$ with $G$ a finite simple group.  Putcha computed the quiver of any regular monoid with exactly $2$ $\J$-classes in terms of decomposing group representations via his method of monoid quivers~\cite{Putcharep3} and so the examples in this subsection could also be handled via his methods.

Let $k$ be an algebraically closed field such that the characteristic of $k$ does not divide the order of $G$.  The simple modules for $\ov G$ are the trivial module $k$ and the simple $kG$-modules $V_1,\dots, V_s$, which become $k\ov G$-modules by having the constants act as the zero map.  Assume that $V_1$ is the trivial $kG$-module.   Since $\ov G = G\cup J$, the relation $\smallsmile$ is empty and hence $X=J$.  So $kX$ can be identified with the permutation module of $G$ coming from the embedding $G\leq S_n$.  Then as a $kG$-module, $M$ is just the result of removing one copy of $V_1$ from the permutation module.  So if the permutation module $kX$ decomposes as $kX=\bigoplus_{i=1}^s m_iV_i$, then there are $m_1-1$ arrows from $k$ to $V_1$ and $m_i$ arrows from $k$ to $V_i$, for $i>1$.   Since there are no directed paths of length $2$ or more in the quiver, it follows that $k\ov G$ is a hereditary algebra.

Assume now that $k$ has characteristic zero.
We provide a complete characterization of the representation type of $k\ov G$ in the case that $G$ acts transitively.   Recall that the \emph{rank} of $G$, denoted $\mathop{\mathrm{rk}}G$, is the number of orbitals of $G$, or equivalently the number of orbits of a point stabilizer~\cite{dixonbook,cameron}.  It is well known that the rank of $G$ is the sum of the squares of the multiplicities of the irreducible constituents of the associated permutation module~\cite{cameron}.  In particular, if the rank is at most $8$, the permutation module must be multiplicity-free. By Gabriel's Theorem~\cite{Benson,assem}, a hereditary algebra has finite representation type if and only if each connected component of the underlying graph of its quiver is Dynkin of type $A$, $D$ or $E$; the algebra has tame representation type if and only if each component is a Euclidean diagram of type $\til A$, $\til D$ or $\til E$.

Let's consider the structure of our quiver.  By transitivity of $G$, the trivial representation of $G$ appears exactly once in the permutation module and so there are no arrows from $k$ to $V_1$.  In general, the quiver consists of isolated points and the connected component $C$ of the trivial $k\ov G$-module $k$, which is a star.  From our description of the quiver, it follows that if $G$ has rank $1$ (and so $n=1$), then $C$ is $A_1$; if $G$ has rank $2$, then $C$ is $A_2$; if $G$ has rank $3$, then $C$ is $A_3$; if $G$ has rank $4$, then $C$ is $D_4$; if the rank of $G$ is $5$, then $C$ is $\til D_4$; and if the rank of $G$ is at least $6$, then $C$ is neither Dynkin nor Euclidean of the above types.  We have thus proved:

\begin{Thm}
Let $G$ be a transitive permutation group and $k$ an algebraically closed field of characteristic zero.  Then:
\begin{enumerate}
\item $k\ov G$ is of finite representation type if and only if $\mathop{\mathrm{rk}} G\leq 4$;
\item $k\ov G$ is of tame representation type if and only if $\mathop{\mathrm{rk}}G=5$;
\item $k\ov G$ is of wild representation type if and only if $\mathop{\mathrm{rk}}G\geq 6$.
\end{enumerate}
\end{Thm}

Since $G$ is $2$-transitive if and only if it has rank $2$, we deduce the following corollary, first proved by Ponizovsky~\cite{ponireptype}.

\begin{Cor}
Let $G$ be a $2$-transitive permutation group.  Then $k\ov G$ is of finite representation type.
\end{Cor}

Notice that computing the quiver of $k\ov G$ is equivalent to decomposing the permutation module $kX$ into irreducibles and so one cannot hope to do better than compute the quiver of a right regular band of groups modulo decomposing group representations.

\subsection{Hsiao's algebra}
Next we want to compute the quiver of Hsiao's algebra~\cite{Saliolafriend}.  To each finite group $G$, Hsiao associates a left regular band of groups $\Sigma_n^G$ (i.e., a regular semigroup in which each left ideal is two-sided).  The group $S_n$ acts by automorphisms on $\Sigma_n^G$ and Hsiao showed that if $G$ is abelian, then the invariant algebra $\mathbb{C}\Sigma_n^{S_n}$ is anti-isomorphic to the descent algebra for $G\wr S_n$ of Mantaci and Reutenauer~\cite{MantReut}, a generalization of Bidigare's result for $S_n$~\cite{BHR,Brown2}; see~\cite{Saliolafriend,Thibon} for the non-abelian case.  Our goal is to compute the quiver for the opposite algebra to $k\Sigma_n^G$ when $k$ has good characteristic.  To be consistent with~\cite{Saliolafriend,Saliola,BHR,Brown1,Brown2} we shall work with $k\Sigma_n^G$ and left modules.  So the vertices of our quiver will be the simple left $k\Sigma_n^G$-modules and the number of arrows from $U$ to $V$ will be $\dim \mathrm{Ext}^1_{k\Sigma_n^G}(U,V)$, where now we work in the category of finitely generated left $k\Sigma_n^G$-modules.  Of course, this quiver is the usual quiver for algebra of the opposite semigroup of $\Sigma_n^G$, which is a right regular band of groups,  hence the results of the previous section apply.  We remind the reader of the construction of Hsiao's left regular band of groups $\Sigma_n^G$.

Fix a finite group $G$ with identity $1_G$ and an integer $n\geq 1$. Set $[n]=\{1,\ldots,n\}$.   An \emph{ordered $G$-partition} of $n$ consists of a sequence
\begin{equation}\label{gpartition}
\tau=((P_1,g_1),\ldots,(P_r,g_r))
\end{equation}
 where $\mathscr P=\{P_1,\ldots, P_r\}$ is a set partition of $\{1,\ldots,n\}$ and $g_1,\ldots,g_r\in G$.   The monoid $\Sigma_n^G$ consists of all ordered $G$-partitions
of $n$ with multiplication given by
\begin{gather*}
((P_1,g_1),\ldots,(P_r,g_r))((Q_1,h_1),\ldots,(Q_s,h_s))=\\
((P_1\cap Q_1,g_1h_1), (P_1\cap Q_2,g_1h_2),\ldots,(P_r\cap Q_1,g_rh_1),\ldots,(P_r\cap Q_s,g_rh_s))
\end{gather*}
where empty intersections are omitted.   In fact, in~\cite{Saliolafriend} Hsiao writes $h_ig_j$ instead of $g_jh_i$, but it is easy to see that our semigroup is isomorphic to his using the inversion in the group $G$.  The identity element of $\Sigma_n^G$ is $([n],1_G)$.  One can compute directly that
\[((P_1,g_1),\ldots,(P_r,g_r))\J((Q_1,h_1),\ldots,(Q_s,h_s))\] if and only if $r=s$ and the set partitions $\{P_1,\ldots,P_r\}$ and  $\{Q_1,\ldots,Q_r\}$ are equal; so $\J$-classes are in bijection with set partitions of $n$.  We write $J_{\mathscr P}$ for the $\J$-class corresponding to a set partition $\mathscr P$.  In fact, the $\J$-order is precisely the usual refinement order on set partitions: $\mathscr P\leq \mathscr Q$ if and only if each block of $\mathscr P$ is contained in a block of $\mathscr Q$.  Let us write $\mathscr P\prec \mathscr Q$ if $\mathscr P$ is covered by $\mathscr Q$ in this ordering, that is, one can obtain $\mathscr Q$ from $\mathscr P$ by joining together two blocks.

An element $\tau$ as per \eqref{gpartition} is idempotent if and only if $g_1=\cdots=g_r=1_G$.  The maximal subgroup at $\tau$ in this case is isomorphic to $G^r$.   The semigroup $\Sigma_n^G$ is a left regular band of groups satisfying the identities $x^{|G|+1}=x$ and $xyx^{|G|}=xy$~\cite{Saliolafriend}.
In particular, if $k$ is an algebraically closed field whose characteristic does not divide $|G|$, then we can apply our techniques to $k\Sigma_n^G$.  So from now on we assume that $k$ is such a field.

Let $X$ be a set.  Define an \emph{$X$-labelled set partition} of $n$ to be a subset $\{(P_1,x_1),\ldots, (P_r,x_r)\}$ of $2^{[n]}\times X$ such that $\{P_1,\ldots, P_r\}$ is a set partition of $n$. We are now ready to describe the quiver of $k\Sigma_n^G$.

\begin{Thm}\label{quiverofgpartitions}
Let $G$ be a finite group and $k$ an algebraically closed field such that the characteristic of $k$ does not divide $|G|$.   Denote by $\Irr G$ the set of simple left $kG$-modules.  Then, for $n\geq 1$, the quiver of $k\Sigma_n^G$ has vertex set the $\Irr G$-labelled set partitions of $n$.   Let $\{(P_1,V_1),\ldots,(P_r,V_r)\}$ be an $\Irr G$-labelled set partition.  We specify the outgoing arrows from this vertex as follows.  Let $U\in \Irr G$ and  $1\leq i\neq j\leq r$.  Then there are  $ \dim \Hom_{kG}(U,V_i\otimes_k V_j)$ (i.e., the multiplicity of $U$ as a composition factor in $V_i\otimes_k V_j$)
arrows from $\{(P_1,V_1),\ldots,(P_r,V_r)\}$ to \[\{(P_i\cup P_j,U)\}\cup \{(P_1,V_1),\ldots, (P_r,V_r)\}\setminus \{(P_i,V_i),(P_j,V_j)\}.\]
\end{Thm}
\begin{proof}
In what follows we do not distinguish between a set partition of $n$ and the corresponding equivalence relation on $[n]$.
The $\J$-classes of $\Sigma_n^G$ are in bijection with set partitions.   If $\mathscr P=\{P_1,\ldots,P_r\}$ is a set partition, the maximal subgroup of the corresponding $\J$-class $J_{\mathscr P}$ is isomorphic to the group $G_{\mathscr P} = G^{[n]/\mathscr P}$.  If $e=((P_{i_1},1_G),\ldots, (P_{i_r},1_G))$ is an idempotent of $J_{\mathscr P}$, then the isomorphism $G_e\to G_{\mathscr P}$ takes $((P_{i_1},g_{i_1}),\ldots, (P_{i_r},g_{i_r}))$ to $f\colon [n]/P\to G$ given by $f(P_i) = g_i$.

Now the simple $kG_{\mathscr P}$-modules are in bijection with $\Irr G$-labelled set partitions via the correspondence \[\{(P_1,V_1),\ldots, (P_r,V_r)\}\}\longmapsto V_1\otimes_k\cdots\otimes_k V_r\] with the tensor product action: $f\cdot v_1\otimes\cdots\otimes v_r = f(P_1)v_1\otimes \cdots \otimes f(P_r)v_r$; see~\cite{curtis}.  It follows that the vertices of the quiver of $k\Sigma_n^G$ can be identified with $\Irr G$-labelled partitions.  In this proof, it will be convenient to work with this ``coordinate-free'' model of the maximal subgroups of $\Sigma_n^G$, rather than fixing an idempotent from each $\J$-class.  We will then change idempotent representatives of the $\J$-classes as is convenient and always use our fixed isomorphisms to the $G_{\mathscr P}$ to translate between simple modules from a given maximal subgroup and $\Irr G$-labelled set partitions.

Let us write $\pi$ for the projection $2^{[n]}\times \Irr G\to 2^{[n]}$.  So $\pi$ takes an $\Irr G$-labelled set partition to its ``underlying'' partition.  Fix $\Irr G$-labelled set partitions $P$ and $Q$ and let $\mathscr P=\pi(P)$, $\mathscr Q=\pi(Q)$.
Since the $\J$-order is given by $J_{\mathscr P}\leq_{\J} J_{\mathscr Q}$ if and only if $\mathscr P\leq \mathscr Q$ in the refinement order,  the results of the previous section show there are no arrows $P\to Q$ unless $\mathscr P<\mathscr Q$.   Fix idempotents $e_{\mathscr P}\in J_{\mathscr P}$ and $e_{\mathscr Q}\in J_{\mathscr Q}$.  Replacing $e_{\mathscr P}$ by $e_{\mathscr Q}e_{\mathscr P}$ if necessary, we may assume that $e_{\mathscr P}\leq_{\R}e_{\mathscr Q}$ and hence $e_{\mathscr P}\in e_{\mathscr Q}J_{\mathscr P}$.  Because $e_{\mathscr P}<_{\Lrel} e_{\mathscr Q}$, as the $\J$ and $\Lrel$-orders coincide in a left regular band of groups, we have in fact $e_{\mathscr P}<e_{\mathscr Q}$ (recall that the idempotents of any semigroup are partially ordered by $e\leq f$ if and only if $ef=fe=e$~\cite{CP,qtheor}).

Suppose first that $\mathscr P<\mathscr Q$ but $\mathscr P\nprec \mathscr Q$.  Let us assume that $\mathscr P=\{P_1,\ldots,P_r\}$ and $\mathscr Q= \{Q_1,\ldots,Q_s\}$.  Let $H=G_{e_{\mathscr P}}$ be the maximal subgroup at the idempotent $e_{\mathscr P}$.  We show that the equivalence relation $\approx$ on $e_{\mathscr Q}J_{\mathscr P}$ has $|H|$ classes.  It will then follow that the bimodule $M$ constructed in the previous section is zero, as it has dimension $|X|-|H|$ where $X=e_{\mathscr Q}J/{\approx}$, and hence there are no arrows in the quiver from any vertex associated to $J_{\mathscr P}$ to any vertex associated to $J_{\mathscr Q}$.  To do this, we show that for any $\gamma\in e_{\mathscr Q}J$, one has that $\gamma\approx e_{\mathscr P}\gamma$.   This will show that each element of $e_{\mathscr Q}J$ is equivalent to an element of $H$.  Since distinct elements of $H$ are never identified under $\approx$, this will establish that ${\approx}$ has $|H|$ classes.

 In the proof, we shall need the following notation.  If $\sigma\in S_r$, then set
\begin{equation}\label{permnotation}
\sigma ((B_1,g_1),\ldots, (B_r,g_r)) = ((B_{\sigma(1)},g_{\sigma(1)}),\ldots,(B_{\sigma(r)},g_{\sigma(r)})).
\end{equation}
Without loss of generality we may assume  $e_{\mathscr P} = ((P_1,1_G),\ldots,(P_r,1_G))$.  The elements of $H$ are then precisely the elements $\tau$ of the form \eqref{gpartition}.  For $\gamma\in J_{\mathscr P}$,  it is easy to see that $e_{\mathscr P}\gamma= \tau$, with $\tau$ as per \eqref{gpartition}, if and only if there is a permutation $\sigma\in S_r$ so that $\gamma =\sigma\tau$ using the notation of \eqref{permnotation}.      The proof that ${\approx}$ has $|H|$ classes relies on two claims.

\begin{Claim}\label{claim1}
Suppose $\gamma_1,\gamma_2\in J_{\mathscr P}$ are such that  $e_{\mathscr P}\gamma_1=e_{\mathscr P}\gamma_2$ and there exists $\rho\in \Sigma_n^G$ with $\rho \ngeq_{\J} J_{\mathscr Q}$ and $\rho\gamma_i=\gamma_i$, $i=1,2$.  Then $e_{\mathscr Q}\gamma_1\approx e_{\mathscr Q}\gamma_2$.
\end{Claim}
\begin{proof}
Because $e_{\mathscr P}<e_{\mathscr Q}$,  if $e_{\mathscr P}\gamma_1=e_{\mathscr P}\gamma_2$, then $e_{\mathscr P}e_{\mathscr Q}\gamma_1=e_{\mathscr P}\gamma_1=e_{\mathscr P}\gamma_2=e_{\mathscr P}e_{\mathscr Q}\gamma_2$.  Next observe that $e_{\mathscr Q}\rho<_{\J} J_{\mathscr Q}$.  Because $e_{\mathscr Q}\gamma_i = e_{\mathscr Q}\rho\gamma_i = e_{\mathscr Q}\rho e_{\mathscr Q}\gamma_i$, for $i=1,2$, in the case $e_{\mathscr Q}\rho\notin J_{\mathscr P}$ it is immediate that $e_{\mathscr Q}\gamma_1\smallsmile e_{\mathscr Q}\gamma_2$.  If $e_{\mathscr Q}\rho\in J_{\mathscr P}$, then $e_{\mathscr Q}\gamma_i= e_{\mathscr Q}\rho \gamma_i \R e_{\mathscr Q}\rho$ by stability.  Since $J_{\mathscr P}$ has a unique $\Lrel$-class, it follows that $e_{\mathscr Q}\gamma_1\H e_{\mathscr Q}\gamma_2$.  Let $f=(e_{\mathscr Q}\gamma_1)^{\omega} = (e_{\mathscr Q}\gamma_2)^{\omega}$.  Then $f\Lrel e_{\mathscr P}$ and so $fe_{\mathscr P}= f$.  Thus we have \[e_{\mathscr Q}\gamma_1 = fe_{\mathscr Q}\gamma_1 = fe_{\mathscr P}e_{\mathscr Q}\gamma_1=fe_{\mathscr P}e_{\mathscr Q}\gamma_2 = fe_{\mathscr Q}\gamma_2 = e_{\mathscr Q}\gamma_2.\]
This proves the claim.
\end{proof}

\begin{Claim}\label{claim2}
Given $(m\ m+1)\in S_r$  and $\alpha\in J_{\mathscr P}$,  there exists $\rho\in \Sigma_n^G$ with $\rho\ngeq_{\J} J_{\mathscr Q}$ and $\rho \alpha=\alpha, \rho (m\ m+1)\alpha = (m\ m+1)\alpha$ where we follow the notation of \eqref{permnotation}.
\end{Claim}
\begin{proof}
Suppose that  $\alpha =  ((P_{j_1},g_1),\ldots,(P_{j_r},g_r))$.  Let $\mathscr P'$ be the partition obtained from $\mathscr P$ by joining $P_{j_m}$ and $P_{j_{m+1}}$.  Since $\mathscr P\prec \mathscr P'$ and $\mathscr P\nprec \mathscr Q$, it follows that $J_{\mathscr P'}\ngeq_{\J} J_{\mathscr Q}$.  Routine computation shows that \[\rho = ((P_{j_1},1_G),\ldots, (P_{j_m}\cup P_{j_{m+1}},1_G),\ldots, (P_{j_r},1_G))\] does the job.
\end{proof}

To complete the proof in the case $\mathscr P\nprec \mathscr Q$, we must show that if  $\tau$ is as in \eqref{gpartition}  and  $\gamma=\sigma\tau\in e_{\mathscr Q}J_{\mathscr P}$ with $\sigma\in S_r$, then $\gamma\approx \tau$.   Since $S_r$ is generated by the consecutive transpositions,  we can connect $\tau$ to $\gamma$ by a sequence $\tau=\alpha_0,\alpha_1,\ldots, \alpha_n=\gamma$ where $\alpha_{i+1}=(m_i\ m_i+1)\alpha_i$ for some $m_i$,  all $i=0,\ldots,n-1$.  Then by Claims~\ref{claim1} and~\ref{claim2}, we have $\tau=e_{\mathscr Q}\alpha_0\approx e_{\mathscr Q}\alpha_1\approx\cdots\approx e_{\mathscr Q}\alpha_n=\gamma$.

Next suppose that $\mathscr P\prec \mathscr Q$.  Then we can order our $\Irr G$-labelled partitions $P$ and $Q$ so that $\mathscr P =\{P_1,\ldots,P_r\}$ and $\mathscr Q = \{P_1,\ldots, P_{r-2},P_{r-1}\cup P_r\}$.    To fix notation, we write
\begin{align*}
P&= \{(P_1,V_1),\ldots, (P_r,V_r)\}\\
Q&= \{(P_1,U_1),\ldots, (P_{r-2},U_{r-2}),(P_{r-1}\cup P_r,U)\}.
\end{align*}
We choose as representatives of  $J_{\mathscr P}$ and $J_{\mathscr Q}$ the respective idempotents $e_{\mathscr P} = ((P_1,1_G),\ldots, (P_r,1_G))$ and $e_{\mathscr Q} = ((P_1,1_G),\ldots, (P_{r-2},1_G),(P_{r-1}\cup P_r,1_G))$.   Notice that $e_{\mathscr P}< e_{\mathscr Q}$.   Then under our isomorphisms $G_{\mathscr P}\cong G_{e_{\mathscr P}}$ and  $G_{\mathscr Q}\cong G_{e_{\mathscr Q}}$ the simple $kG_{e_{\mathscr P}}$-module corresponding to $P$ is $V_1\otimes_k\cdots\otimes_k V_r$ and the simple $k G_{e_{\mathscr Q}}$-module corresponding to $Q$ is $U_1\otimes_k\cdots\otimes_k U_{r-2}\otimes_k U$.

Since $J_{\mathscr Q}$ covers $J_{\mathscr P}$ in the $\J$-order, the equivalence relation $\approx$ identifies no elements.  Therefore, $X=ke_{\mathscr Q}J_{\mathscr P}$.  Now it is not hard to compute that $e_{\mathscr Q}J_{\mathscr P}$ consists of all elements of the form $\tau$ and $(r-1\ r)\tau$ with $\tau$ as per \eqref{gpartition}.  Moreover, $e_{\mathscr P}\tau = \tau = e_{\mathscr P}(r-1\ r)\tau$. Therefore, the bimodule $M$ from the previous section  has a basis consisting of the elements of the form $(r-1\ r)\tau -\tau$ with $\tau$ as above. It is immediate that as a right $kG_{e_{\mathscr P}}$-module, $M$ is isomorphic to the regular module since if $\rho\in G_{e_{\mathscr P}}$ then $((r-1\ r)\tau) \rho = (r-1\ r)(\tau\rho)$.  On the other hand if
\begin{equation}\label{onemoretime}
\lambda = ((P_1,h_1),\ldots, (P_{r-2},h_{r-2}),(P_{r-1}\cup P_r,h))\in G_{e_{\mathscr Q}},
\end{equation}
then  $\lambda\big((r-1\ r)\tau -\tau\big)$ is $(r-1\ r)\beta - \beta$ where
\begin{equation}\label{betadude}
\beta = ((P_1,h_1g_1),\ldots, (P_{r-2},h_{r-2}g_{r-2}),(P_{r-1},hg_{r-1}),(P_r,hg_r))
\end{equation}
 as can be verified by direct computation.

Since $M$ is regular as a right $kG_{e_{\mathscr P}}$-module,  \[M\otimes_{kG_{e_{\mathscr P}}} (V_1\otimes_k\cdots\otimes_k V_r)\cong V_1\otimes_k\cdots\otimes_k V_r\] as a vector space.    By \eqref{betadude}, the action of $\lambda\in G_{e_{\mathscr Q}}$ on $V_1\otimes_k\cdots \otimes_k V_r$ is given on elementary tensors by \[\lambda\cdot v_1\otimes\cdots \otimes v_r = h_1v_1\otimes\cdots\otimes h_{r-2}v_{r-2}\otimes hv_{r-1}\otimes hv_r\] with $\lambda$ as per \eqref{onemoretime}.

Since the tensor product distributes over sums, it follows that the multiplicity of $U_1\otimes_k\cdots \otimes_k U_{r-2}\otimes U$  as a composition factor of $V_1\otimes_k\cdots\otimes_k V_r$  is $0$ unless $U_i=V_i$ for $1\leq i\leq r-2$, in which case it is the multiplicity of $U$ as a constituent in $V_{r-1}\otimes_k V_r$.  This completes the proof.
\end{proof}

For example, if $G$ is trivial, then the quiver of $k\Sigma_n^G$ is just the Hasse diagram of the lattice of set partitions of $n$, as was first proved by Saliola~\cite{Saliola} and Schocker~\cite{schocker}.  Suppose now that $G$ is a finite abelian group and the characteristic of $k$ does not divide $|G|$.  Let $\wh{G}$ be the dual group of $G$, that is, $\Hom_{\ZZ}(G,k^*)$.  Of course, $\wh{G}\cong G$ and $\Irr G\cong \wh{G}$.  The tensor product of representations corresponds to the product of the characters.  Thus the quiver of $k\Sigma^G_n$ has vertices all $\wh{G}$-labelled partitions of $n$.  A vertex $\{(P_1,\chi_1),\ldots,(P_r,\chi_r)\}$ has $\binom{r}{2}$ outgoing arrows:  for $1\leq i\neq j\leq r$,  there is an arrow from $\{(P_1,\chi_1),\ldots,(P_r,\chi_r)\}$ to $\{(P_i\cup P_j,\chi_i\cdot \chi_j)\}\cup \{(P_1,\chi_1),\ldots,(P_r,\chi_r)\}\setminus \{(P_i,\chi_i),(P_j,\chi_j)\}$.  In particular, the quiver is acyclic with no multiple arrows between vertices.

\section*{Acknowledgments}
We would like to thank Franco Saliola for several insightful discussions, as well as for pointing out to us Hsiao's work~\cite{Saliolafriend}.  He also provided helpful comments on a preliminary draft of this paper.

\bibliographystyle{abbrv}
\bibliography{standard2}
\end{document}